\documentclass{amsart}
\usepackage{amssymb}
\usepackage{verbatim}
\usepackage{array}
\usepackage{latexsym}
\usepackage{enumerate}
\usepackage{tikz}
\usepackage{float}
\usepackage{amsmath}
\usepackage{amsfonts}
\usepackage{amsthm}
\usepackage{color}
\usepackage[english]{babel}

\newtheorem{theorem}{Theorem}[section]
\newtheorem{proposition}[theorem]{Proposition}
\newtheorem{lemma}[theorem]{Lemma}

\newtheorem{remark}[theorem]{Remark}

\def\cA{\mathcal A}

\def\cC{\mathcal C}

\def\cH{\mathcal H}

\def\cX{\mathcal X}

\def\fqs{\mathbb{F}_{q^2}}

\def\PG{{\rm{PG}}}

\def\SL{{\rm{SL}}}
\def\ord{\mbox{\rm ord}}

\def\det{\mbox{\rm det}}

\def\Ker{\mbox{\rm Ker}}

\def\bq{{\bar q}}

\newcommand{\PSL}{\mbox{\rm PSL}}
\newcommand{\PGL}{\mbox{\rm PGL}}

\newcommand{\PSU}{\mbox{\rm PSU}}
\newcommand{\PGU}{\mbox{\rm PGU}}

\newcommand{\aut}{\mbox{\rm Aut}}

\newcommand{\diag}{\mbox{\rm diag}}

\title{Quotients of the Hermitian curve from subgroups of $\PGU(3,q)$ without fixed points or triangles}
\date{}
\author{Maria Montanucci and Giovanni Zini }

\begin{document}



\begin{abstract}
In this paper we deal with the problem of classifying the genera of quotient curves $\mathcal{H}_q/G$, where $\mathcal{H}_q$ is the $\mathbb{F}_{q^2}$-maximal Hermitian curve and $G$ is an automorphism group of $\cH_q$.
The groups $G$ considered in the literature fix either a point or a triangle in the plane $\PG(2,q^6)$.
In this paper, we give a complete list of genera of quotients $\cH_q/G$, when $G \leq \aut(\cH_q) \cong \PGU(3,q)$ does not leave invariant any point or triangle in the plane.
As a result, the classification of subgroups $G$ of $\PGU(3,q)$ satisfying this property is given up to isomorphism.
\end{abstract}

\maketitle

\begin{small}



\footnote{
This research was partially supported by Ministry for Education, University
and Research of Italy (MIUR)
 and by the Italian National Group for Algebraic and Geometric Structures
and their Applications (GNSAGA - INdAM).

URL: Maria Montanucci (maria.montanucci@unibas.it), Giovanni Zini (giovanni.zini@unimib.it).
}

\end{small}

\section{Introduction}

Let $\mathbb{F}_q$ be a finite field of order $q$ and $\cX$ be a projective, irreducible, non-singular algebraic curve of genus $g$ defined over $\mathbb{F}_{q}$. The problems of determining the maximum number of points over $\mathbb{F}_{q}$ that $\cX$ can have and finding examples of algebraic curves $\cX$ with many rational points have been important, not only from the theoretic perspective, but also for applications in Coding Theory; see for instance \cite{GS,vdG,vdG2}.
The Hasse-Weil Theorem provides an upper bound for the number of rational points $|\cX(\mathbb{F}_q)|$ that a curve $\cX$ defined over $\mathbb{F}_q$ can have, namely $|\cX(\mathbb{F}_q)| \leq q+1+2g \sqrt{q}$. If $|\cX(\mathbb{F}_q)| =q+1+2g \sqrt{q}$, then the curve $\cX$ is said to be $\mathbb{F}_q$-maximal. Clearly, $\cX$ can be $\mathbb{F}_q$-maximal only if either $g$ is zero or $q$ is a square.
A natural question in this context is: over a finite field $\mathbb{F}_{q^2}$ of square cardinality, which non-negative integers $g$ can be realized as the genera of maximal curves over $\mathbb{F}_{q^2}$?

A leading example of a maximal curve is the Hermitian curve $\cH_q$ over $\mathbb{F}_{q^2}$, where $q$ is a power of a prime $p$. It is defined as the non-singular plane curve admitting one of the following birational equivalent plane models: $X^{q+1}+Y^{q+1}+Z^{q+1}=0$ and $X^qZ+XZ^q=Y^{q+1}$.

For fixed $q$, the curve $\cH_q$ has the largest genus $g(\cH_q)=q(q-1)/2$ that an $\mathbb{F}_{q^2}$-maximal curve can have; see \cite{GSX} and the references therein. The full automorphism group $\aut(\cH_q)$ is isomorphic to $\PGU(3,q)$, the group of projectivities of ${\rm PG}(2,q^2)$ commuting with the unitary polarity associated with $\cH_q$. The automorphism group $\aut(\cH_q)$ is extremely large with respect to the value $g(\cH_q)$. Indeed it is know that the Hermitian curve is the unique curve of  genus $g \geq 2$ up to isomorphisms admitting an automorphism group of order at least $16g^4$. 

By a result commonly referred to as the Kleiman-Serre covering result, see \cite{KS} and \cite[Proposition 6]{L}, a curve defined over $\mathbb{F}_{q^2}$ which is $\mathbb{F}_{q^2}$-covered  by an $\mathbb{F}_{q^2}$-maximal curve is $\mathbb{F}_{q^2}$-maximal as well. In particular, $\mathbb{F}_{q^2}$-maximal curves can be obtained as Galois $\mathbb{F}_{q^2}$-subcovers of an $\mathbb{F}_{q^2}$-maximal curve $\cX$, that is, as quotient curves $\cX/G$ where $G$ is a finite $\mathbb{F}_{q^2}$-automorphism group of $\cX$. Since $\aut(\cH_q)$ is large and has plenty of non-conjugated subgroups, it seems to be natural to construct maximal curves as Galois $\mathbb{F}_{q^2}$-subcovers $\cH_q/G$ of the Hermitian curve $\cH_q$ with $G \leq \PGU(3,q)$. Indeed, most of the known maximal curves are constructed in this way; see for instance \cite{GSX,CKT2,GHKT2,MZ2,MZ} and the references therein.
The most significant cases treated in the literature are the following:
\begin{itemize}
\item $G$ fixes an $\mathbb{F}_{q^2}$-rational point of $\cH_q$; see \cite{BMXY,GSX,AQ}.
\item $G$ normalizes a Singer subgroup of $\cH_q$ acting on three $\mathbb{F}_{q^6}$-rational points of $\cH_q$; see \cite{GSX,CKT}.
\item $G$ has prime order; see \cite{CKT2}.
\item $G$ fixes an $\mathbb{F}_{q^2}$-rational point off $\cH_q$ see \cite{MZ}. 
\end{itemize}
From these results, in order to obtain the complete list of genera of quotients $\cH_q/G$ of the Hermitian curve $\cH_q$, only the following cases for $G\leq\aut(\cH_q)$ still have to be described and characterized completely:
\begin{enumerate}
\item $G$ fixes a self-polar triangle in ${\rm PG}(2,q^2)$,
\item $G$ fixes a point $P \in \cH_q(\mathbb{F}_{q^2})$, $p=2$ and $|G|=p^\ell d$ where $p^\ell \leq q$ and $d \mid (q-1)$,
\item $G$ fixes an $\mathbb{F}_{q^2}$-rational point $P\notin\cH_q$,
\item $G$ does not fix any point or triangle in $\PG(2,q^6)$.
\end{enumerate}
The structure of maximal subgroups of $\PGU(3,q)$ fixing a point or a triangle in $\PG(2,q^6)$ is known, see \cite{M,H}; yet, the structure of subgroups $\PGU(3,q)$ fixing no points nor triangles is not explicitly described when $3 \mid (q+1)$, that is, when $\PGU(3,q)\ne\PSU(3,q)$.
Also, the groups $G\leq\PGU(3,q)$ considered in the literature always have a fixed point or a fixed triangle in $\PG(2,q^6)$ (with the single exception of \cite[Proposition 5.1]{MZ2}).

In this paper we give the complete list of genera of quotients $\cH_q/G$ when $G$ does not fix any point or triangle in $\PG(2,q^6)$. To this aim we also provide the list of such subgroups of $\PGU(3,q)$ up to isomorphisms, and we give explicitly the list of maximal subgroups of $\PGU(3,q)$, which was already provided in \cite{M,H} for $3\nmid(q+1)$.

The organization of the paper is the following. Section \ref{sec:preliminari} provides preliminary results on the structure of $\PGU(3,q)$ and $\PSU(3,q)$, and on the Riemann-Hurwitz formula applied to quotients of the Hermitian curve. Section \ref{sec:sottogruppimassimaliPGU} contains a description of the maximal subgroups of $\PGU(3,q)$ in relation to the maximal subgroups of $\PSU(3,q)$. Section \ref{sec:generi} applies the results of Sections \ref{sec:preliminari} and \ref{sec:sottogruppimassimaliPGU} to compute the genus of $\cH_q/G$ whenever $G$ does not fix any point or triangle.
Section \ref{sec:geometria} lists the subgroups of $\PGU(3,q)$ acting in the plane without fixed points or triangles.

\section{Preliminary results}\label{sec:preliminari}

Throughout the paper, $q$ will denote a power $p^n$ of a prime $p$.
Maximal subgroups of $\PSU(3,q)$ were explicitly described in \cite{M} and \cite{H}, see also \cite[Theorem A.10]{HKT}. 

The structure of maximal subgroups $M$ of the automorphism group $\aut(\cH_q)\cong\PGU(3,q)$ of the Hermitian curve $\cH_q$ which fix a point or a triangle in $\PG(2,\bar{\mathbb{F}}_{q^2})$ is known and the genera of quotients $\cH_q/G$ with $G\leq M$ have been deeply investigated; see for instance \cite{BMXY,CKT,CKT2,GSX,MZ2,MZ}. Such maximal subgroups are the following, up to conjugation.
\begin{itemize}
\item[(I)] The stabilizer $M_1$ of a point $P\in\cH_q(\mathbb{F}_{q^2})$. The group $M_1$ has order $q^3(q^2-1)$ and is a semidirect product of its unique Sylow $p$-subgroup of order $q^3$ and a cyclic group $C$ of order $q^2-1$. The group $C$ fixes also another point $Q\ne P$ with $Q\in\cH_q(\mathbb{F}_{q^2})$.
\item[(II)] The stabilizer $M_2$ of a pole-polar pair $(P,\ell)$ with respect to the unitary polarity associated to $\cH_q$, with $P\in\PG(2,q^2)\setminus\cH_q$ and $P\notin\ell$. The group $M_2$ has order $q(q+1)^2(q-1)$ and is a semidirect product isomorphic to $\SL(2,q)\rtimes C_{q+1}$, where $C_{q+1}$ is cyclic of order $q+1$.
\item[(III)] The stabilizer $M_3$ of a self-polar triangle $T=\{P_1,P_2,P_3\}$ with respect to the unitary polarity associated to $\cH_q$, with $P_i\in\PG(2,q^2)\setminus\cH_q$. The group $M_3$ has order $6(q+1)^2$ and is a semidirect product of an abelian group $C_{q+1}\times C_{q+1}$ fixing $T$ pointwise and a symmetric group $S_3$ acting faithfully on $T$.
\item[(IV)] The stabilizer $M_4$ of a triangle $T=\{P_1,P_2,P_3\}\subset\cH_q(\mathbb{F}_{q^6})\setminus\cH_q(\mathbb{F}_{q^2})$ which is invariant under the Frobenius automorphism $(X,Y,Z)\mapsto (X^{q^2},Y^{q^2},Z^{q^2})$. The group $M_4$ has order $3(q^2-q+1)$ and is a semidirect product $C_{q^2-q+1}\rtimes C_3$, where $C_{q^2-q+1}$ is a Singer group acting semiregularly on $\PG(2,q^2)$ and fixing $T$ pointwise, and $C_3$ has a unique orbit on $T$.
\end{itemize}

For any $i\in\{1,2,3,4\}$, the intersection $\PSU(3,q)\cap M_i$ of $\PSU(3,q)$ with the maximal subgroup $M_i$ listed above is a maximal subgroup of $\PSU(3,q)$ with index $\gcd(3,q+1)$ in $M_i$; see \cite{M} and \cite{H}.

\begin{theorem}{\rm{(}\cite{M,H}\rm{)}}\label{MaximalSubgroupsNotFixingPSU}
If $p$ is odd, then the maximal subgroups of $\PSU(3,q)$ not fixing a point nor a triangle are the following:
\begin{itemize}
\item[(i)] the Hessian groups of order $216$ when $9\mid(q+1)$, and of order $72$ when $3\mid(q+1)$ and $9\nmid(q+1)$;
\item[(ii)] ${\rm PGL}(2,q)$ preserving a conic;
\item[(iii)] ${\rm PSL(2,7)}$ when $p=7$ or $-7$ is not a square in $\mathbb{F}_q$;
\item[(iv)] the alternating group ${\rm A}_6$ when either $p=3$ and $n$ is even, or $5$ is a square in $\mathbb{F}_q$ but $\mathbb{F}_q$ contains no cube roots of unity;
\item[(v)] a group of order $720$ containing the alternating group ${\rm A}_6$ as a normal subgroup when $p=5$ and $n$ is odd;
\item[(vi)] the alternating group ${\rm A}_7$ when $p=5$ and $n$ is odd;
\item[(vii)] $\PSU(3,p^m)$ where $n/m$ is an odd prime different from $3$; 
\item[(viii)] subgroups containing $\PSU(3,p^m)$ as a normal subgroup of index $3$, when $m\mid n$, $n/m=3$, and $3\mid(q+1)$. 
\end{itemize}
If $p=2$, then the maximal subgroups of $\PSU(3,q)$ not fixing a point nor a triangle are the following:
\begin{itemize}
\item[(vii$^\prime$)] $\PSU(3,2^m)$ where $m\mid n$ and $n/m$ is an odd prime different from $3$; when $m=1$, this is the Hessian group of order $72$;
\item[(viii$^\prime$)] subgroups containing $\PSU(3,2^m)$ as a normal subgroup of index $3$, where $n/m=3$ and $3\mid(q+1)$. For $m=1$ this is the Hessian group of order $216$;
\item[(ix$^\prime$)]  a group of order $36$, which exists as a maximal subgroup when $n=1$.
\end{itemize}
\end{theorem}

The following lemma recalls how an element of $\PGU(3,q)$ of a given order acts on $\PG(2,\mathbb{K})$, and in particular on $\cH_q(\mathbb{F}_{q^2})$; for the usual terminology about collineations of projective planes; see \cite{HP}.

\begin{lemma}{\rm{(}\cite[Lemma 2.2]{MZ}\rm{)}}\label{classificazione}
For a nontrivial element $\sigma\in\PGU(3,q)$, one of the following cases holds.
\begin{itemize}
\item[(A)] ${\rm ord}(\sigma)\mid(q+1)$ and $\sigma$ is a homology whose center $P$ is a point off $\cH_q$ and whose axis $\ell$ is a chord of $\cH_q(\mathbb{F}_{q^2})$ such that $(P,\ell)$ is a pole-polar pair with respect to the unitary polarity associated to $\cH_q(\mathbb{F}_{q^2})$.
\item[(B)] ${\rm ord}(\sigma)$ is coprime to $p$ and $\sigma$ fixes the vertices $P_1,P_2,P_3$ of a non-degenerate triangle $T$.
\begin{itemize}
\item[(B1)] The points $P_1,P_2,P_3$ are $\fqs$-rational, $P_1,P_2,P_3\notin\cH_q$ and the triangle $T$ is self-polar with respect to the unitary polarity associated to $\cH_q(\mathbb{F}_{q^2})$. Also, $o(\sigma)\mid(q+1)$.
\item[(B2)] The points $P_1,P_2,P_3$ are $\fqs$-rational, $P_1\notin\cH_q$, $P_2,P_3\in\cH_q$. 
     Also, $o(\sigma)\mid(q^2-1)$ and $\ord(\sigma)\nmid(q+1)$.
\item[(B3)] The points $P_1,P_2,P_3$ have coordinates in $\mathbb{F}_{q^6}\setminus\mathbb{F}_{q^2}$, $P_1,P_2,P_3\in\cH_q$. 
    Also, $\ord(\sigma)\mid (q^2-q+1)$.
\end{itemize}
\item[(C)] ${\rm ord}(\sigma)=p$ and $\sigma$ is an elation whose center $P$ is a point of $\cH_q$ and whose axis $\ell$ is a tangent of $\cH_q(\mathbb{F}_{q^2})$; here $(P,\ell)$ is a pole-polar pair with respect to the unitary polarity associated to $\cH_q(\mathbb{F}_{q^2})$.
\item[(D)] ${\rm ord}(\sigma)=p$ with $p\ne2$, or ${\rm ord}(\sigma)=4$ and $p=2$. In this case $\sigma$ fixes an $\fqs$-rational point $P$, with $P \in \cH_q$, and a line $\ell$ which is a tangent  of $\cH_q(\mathbb{F}_{q^2})$; here $(P,\ell)$ is a pole-polar pair with respect to the unitary polarity associated to $\cH_q(\mathbb{F}_{q^2})$.
\item[(E)] $p\mid{\rm ord}(\sigma)$, $p^2\nmid{\rm ord}(\sigma)$, and ${\rm ord}(\sigma)\ne p$. In this case $\sigma$ fixes two $\fqs$-rational points $P,Q$, 
     with $P\in\cH_q$, $Q\notin\cH_q$. 
\end{itemize}
\end{lemma}

Throughout the paper, a nontrivial element of $\PGU(3,q)$ is said to be of type (A), (B), (B1), (B2), (B3), (C), (D), or (E), as given in Lemma \ref{classificazione}; $G$ always stands for a subgroup of $\PGU(3,q)$.

Every subgroup $G$ of $\PGU(3,q)$ produces a quotient curve $\cH_q/G$, and the cover $\cH_q\rightarrow\cH_q/G$ is a Galois cover defined over $\mathbb{F}_{q^2}$  where the degree of the different divisor $\Delta$ is given by the Riemann-Hurwitz formula \cite[Theorem 3.4.13]{Sti},
$$\Delta=(2g(\cH_q)-2)-|G|(2g(\cH_q/G)-2).$$ 

On the other hand, $\Delta=\sum_{\sigma\in G\setminus\{id\}}i(\sigma)$, where $i(\sigma)\geq0$ is given by the Hilbert's different formula \cite[Thm. 3.8.7]{Sti}, namely
\begin{equation}\label{contributo}
\textstyle{i(\sigma)=\sum_{P\in\cH_q(\bar{\mathbb{F}}_{q^2})}v_P(\sigma(t)-t),}
\end{equation}
where $t$ is a local parameter at $P$.

By analyzing the geometric properties of the elements $\sigma \in \PGU(3,q)$, it turns out that there are only few possibilities for $i(\sigma)$, as stated in the following theorem.

\begin{theorem}\label{caratteri}{\rm{(}\cite[Theorem 2.7]{MZ}\rm{)}}
For a nontrivial element $\sigma\in \PGU(3,q)$ one of the following cases occurs.
\begin{enumerate}
\item If $\ord(\sigma)=2$ and $2\mid(q+1)$, then $\sigma$ is of type {\rm(A)} and $i(\sigma)=q+1$.
\item If $\ord(\sigma)=3$, $3 \mid(q+1)$ and $\sigma$ is of type {\rm(B3)}, then $i(\sigma)=3$.
\item If $\ord(\sigma)\ne 2$, $\ord(\sigma)\mid(q+1)$ and $\sigma$ is of type {\rm(A)}, then $i(\sigma)=q+1$.
\item If $\ord(\sigma)\ne 2$, $\ord(\sigma)\mid(q+1)$ and $\sigma$ is of type {\rm(B1)}, then $i(\sigma)=0$.
\item If $\ord(\sigma)\mid(q^2-1)$ and $\ord(\sigma)\nmid(q+1)$, then $\sigma$ is of type {\rm(B2)} and $i(\sigma)=2$.
\item If $\ord(\sigma)\ne3$ and $\ord(\sigma)\mid(q^2-q+1)$, then $\sigma$ is of type {\rm(B3)} and $i(\sigma)=3$.
\item If $p=2$ and $\ord(\sigma)=4$, then $\sigma$ is of type {\rm(D)} and $i(\sigma)=2$.
\item If $\ord(\sigma)=p$, $p \ne2$ and $\sigma$ is of type {\rm(D)}, then $i(\sigma)=2$.
\item If $\ord(\sigma)=p$ and $\sigma$ is of type {\rm(C)}, then $i(\sigma)=q+2$.
\item If $\ord(\sigma)\ne p$, $p\mid\ord(\sigma)$ and $\ord(\sigma)\ne4$, then $\sigma$ is of type {\rm(E)} and $i(\sigma)=1$.
\end{enumerate}
\end{theorem}



\section{On maximal subgroups of $\PGU(3,q)$}\label{sec:sottogruppimassimaliPGU}

Since we were not able to find references for the complete list of maximal subgroups of $\PGU(3,q)$, we determine it in this section when $3 \mid (q+1)$, that is, when $\PGU(3,q)\ne\PSU(3,q)$.
The result is the following.

\begin{theorem} \label{MaximalSubgroupsPGU}
Let $q=p^n$ be a prime power such that $3 \mid (q+1)$. Then the following is the list of maximal subgroups of $\PGU(3,q)$.
\begin{enumerate}
\item The stabilizer $M_1$ of a point $P\in\cH_q(\mathbb{F}_{q^2})$. The group $M_1$ has order $q^3(q^2-1)$ and is a semidirect product of its unique Sylow $p$-subgroup of order $q^3$ and a cyclic group $C$ of order $q^2-1$. The group $C$ fixes also another point $Q\ne P$ with $Q\in\cH_q(\mathbb{F}_{q^2})$.
\item The stabilizer $M_2$ of a pole-polar pair $(P,\ell)$ with respect to the unitary polarity associated to $\cH_q$, with $P\in\PG(2,q^2)\setminus\cH_q$ and $P\notin\ell$. The group $M_2$ has order $q(q+1)^2(q-1)$ and is a semidirect product isomorphic to $\SL(2,q)\rtimes C_{q+1}$, where $C_{q+1}$ is cyclic of order $q+1$.
\item The stabilizer $M_3$ of a self-polar triangle $T=\{P_1,P_2,P_3\}$ with respect to the unitary polarity associated to $\cH_q$, with $P_i\in\PG(2,q^2)\setminus\cH_q$. The group $M_3$ has order $6(q+1)^2$ and is a semidirect product of an abelian group $C_{q+1}\times C_{q+1}$ fixing $T$ pointwise and a symmetric group $S_3$ acting faithfully on $T$.
\item The stabilizer $M_4$ of a triangle $T=\{P_1,P_2,P_3\}\subset\cH_q(\mathbb{F}_{q^6})\setminus\cH_q(\mathbb{F}_{q^2})$ which is invariant under the Frobenius automorphism $(X,Y,Z)\mapsto (X^{q^2},Y^{q^2},Z^{q^2})$. The group $M_4$ has order $3(q^2-q+1)$ and is a semidirect product $C_{q^2-q+1}\rtimes C_3$, where $C_{q^2-q+1}$ is a Singer group acting semiregularly on $\PG(2,q^2)$ and fixing $T$ pointwise, and $C_3$ has a unique orbit on $T$.
\item The normal subgroup $\PSU(3,q)$ of index $3$ in $\PGU(3,q)$.
\item The Hessian group $H_{216} \cong \PGU(3,2)$, if $p$ is odd and $9 \nmid (q+1)$.
\item $\PGU(3,p^m)$ where $m \mid n$ and $n/m$ is an odd prime different from $3$.
\end{enumerate}
\end{theorem}

In order to prove Theorem \ref{MaximalSubgroupsPGU}, we proceed with a case-by-case analysis on maximal subgroups $M$ of $\PGU(3,q)$ such that $M\cap\PSU(3,q)$ is isomorphic to one of the maximal subgroups of $\PSU(3,q)$ not fixing a point nor a triangle, listed in Theorem \ref{MaximalSubgroupsNotFixingPSU}.

\subsection{Hessian groups}

\begin{proposition}\label{HessianInPSU}
Let $q=p^n$ be a power of an odd prime $p$ such that $9\mid(q+1)$. Let $H\cong\PGU(3,2)$ be a maximal subgroup of $\PSU(3,q)$ isomorphic to the Hessian group of order $216$, and $M$ be a maximal subgroup of $\PGU(3,q)$ containing $H$. Then $M\cong\PSU(3,q)$.
\end{proposition}

\begin{proof}
Since $H\leq M\cap\PSU(3,q)\leq \PSU(3,q)$ and $H$ is maximal in $\PSU(3,q)$, either $M\cap\PSU(3,q)=H$ or $M\cap\PSU(3,q)=\PSU(3,q)$.
As $\PSU(3,q)$ is maximal in $\PGU(3,q)$, the claim is proved once that $M\cap\PSU(3,q)=\PSU(3,q)$.
Assume by contradiction that $M \cap \PSU(3,q)=H$.
This implies $M\not\subseteq\PSU(3,q)$ as $M$ is maximal in $\PGU(3,q)$. Also, $\PSU(3,q)$ is maximal and normal in $\PGU(3,q)$ with index $3$. Then $\PGU(3,q)=\langle M,\PSU(3,q) \rangle = \PSU(3,q)M$ and
\begin{equation}\label{trevolte}
3\cdot|\PSU(3,q)|=|\PGU(3,q)|=|\PSU(3,q)M|=\frac{|\PSU(3,q)|\cdot|M|}{|H|},
\end{equation}
so that $|M|=3|H|=648$.
By direct checking with MAGMA \cite{MAGMA}, there are exactly $3$ groups of order $648$ containing a subgroup of order $216$ isomorphic to $\PGU(3,2)$, namely $SmallGroup(648,i)$ with $i \in \{702,756,757\}$.
\begin{itemize}
\item {$M\cong SmallGroup(648,702)$.} In this case $M$ has a normal subgroup of order $3$, say $\langle \alpha \rangle$; hence, $M$ acts on the points fixed by $\alpha$. By Lemma \ref{classificazione}, $\alpha$ is either of type (A), or (B1), or (B3); hence, either $M$ fixes a pole-polar pair $(P,\ell)$ with $P\in\PG(2,q^2)\setminus\cH_q$, or $M$ stabilizes a self-polar triangle $\{P_1,P_2,P_3\}\subseteq\PG(2,q^2)\setminus\cH_q$, or $M$ stabilizes a Frobenius-invariant triangle $\{P_1,P_2,P_3\}\subseteq\cH_q(\mathbb{F}_{q^6})\setminus\cH_q(\mathbb{F}_{q^2})$. Then $M$ is properly contained in one of the maximal subgroups $M_1,M_2,M_3,M_4$ described in Theorem \ref{MaximalSubgroupsPGU}, a contradiction to the maximality of $M$.
\item {$M\cong SmallGroup(648,756)$.} In this case $M$ has a normal subgroup of order $2$, say $\langle \alpha \rangle$; hence, $M$ acts on the points fixed by $\alpha$. By Lemma \ref{classificazione}, $\alpha$ is of type (A) and fixes a point $P \in PG(2,q^2) \setminus \cH_q$. Thus, $M$ fixes $P$ and is properly contained in the maximal subgroup $M_2$ of $\PGU(3,q)$, a contradiction.
\item {$M\cong SmallGroup(648,757)$.} In this case $M$ has a normal subgroup of order $2$, and a contradiction follows as in the previous case.
\end{itemize}
Then $M\cap\PSU(3,q)=\PSU(3,q)$ and the claim is proved.
\end{proof}


\begin{proposition} \label{maxpsu32}
Let $q=p^n$ be a power of an odd prime $p$ such that $3 \mid (q+1)$ and $9 \nmid (q+1)$. Let $H\cong\PSU(3,2)$ be a maximal subgroup of $\PSU(3,q)$ isomorphic to the Hessian group of order $72$, and $M$ be a maximal subgroup of $\PGU(3,q)$ containing $H$. Then $M \cong \PGU(3,2)$ is isomorphic to the Hessian group of order $216$.
\end{proposition}

\begin{proof}
Arguing as in the proof of Theorem 9 in \cite{M}, one can explicitly construct a subgroup $M$ of $\PGL(3,q^2)$ containing $H$ as a normal subgroup of index $3$ and isomorphic to the Hessian group $\PGU(3,2)$ of order $216$.
By direct computation, the generators of $M$ leave invariant the Hermitian curve $\cH_q$ in its Fermat equation $X^{q+1}+Y^{q+1}+Z^{q+1}=0$; hence, $M$ is a subgroup of $\PGU(3,q)$.
Also, $M$ is not contained in $\PSU(3,q)$.
Now we show that $M$ is maximal in $\PGU(3,q)$.

Suppose by contradiction there there exists a proper subgroup $M^\prime$ of $\PGU(3,q)$ containing $M$ properly.
By \cite[Theorem 9]{M}, $M$ is the unique tame subgroup of $\PGU(3,q)$ such that $M$ does not leave invariant a point, or a line, or a triangle and $M$ contains homologies of order $3$.
Since $M\subseteq M^\prime$, also $M^\prime$ does not leave invariant a point, or a line, or a triangle and $M^\prime$ contains homologies of order $3$. Then $M^\prime$ contains $p$-elements, which are clearly elements of $\PSU(3,q)$ because $p\nmid[\PGU(3,q):\PSU(3,q)]$.
From \cite[Theorem 18]{M}, $M^\prime$ contains elements of type (C) because $M^\prime$ cannot contain only $p$-elements of type (D). From \cite[Theorem 28]{M}, $M^\prime$ contains the entire $\PSU(3,q)$; since $M^\prime$ is not contained in $\PSU(3,q)$, this implies $M^\prime=\PGU(3,q)$, a contradiction.
\end{proof}

\subsection{$\PGL(2,q)$ preserving a conic}

\begin{proposition} \label{maxPGLconica}
Let $q=p^n$ be a power of an odd prime $p$ with $3 \mid (q+1)$. Let $H$ be a maximal subgroup of $\PSU(3,q)$ such that $H \cong \PGL(2,q)$ and $H$ fixes a conic $\cC$. Let $M$ be a maximal subgroup of $\PGU(3,q)$ containing $H$. Then $M \cong \PSU(3,q)$.
\end{proposition}

\begin{proof}
Since $H\leq M\cap\PSU(3,q)\leq\PSU(3,q)$, we have either $M\cap\PSU(3,q)=H$ or $M\cap\PSU(3,q)=\PSU(3,q)$, and the claim is proved once that $M\cap\PSU(3,q)=\PSU(3,q)$.
Assume by contradiction that $M\cap\PSU(3,q)=H$. As in the proof of Proposition \ref{HessianInPSU}, Equation \eqref{trevolte} holds and $|M|=3\cdot|\PGL(2,q)|$.
Since $\PSU(3,q)$ is normal in $\PGU(3,q)$, also $H$ is normal in $M$. Hence $M$ is a degree $3$ normal extension of $\PGL(2,q)$.
From \cite[Remark 6.12]{GKeven}, $M$ is a semidirect product $M=\PGL(2,q) \rtimes C_3$ where $C_3=\langle \alpha \rangle$ is a cyclic group of order $3$ and $3 \mid n$.
Let $I=\cH_q\cap\cC$ be the set of the $q+1$ points of intersection between $\cH_q$ and $\cC$. The set $I$ is the unique orbit of size $q+1$ of $H$ on $\cH_q$; see \cite[Lemma 3.1]{CE}.
Since $M$ normalizes $H$, $M$ acts on the set of orbits of $H$ on $\cH_q$ having the same size; hence, $M$ acts on $I$.
By Lemma \ref{classificazione}, $\alpha$ is either of type (A), or of type (B1), or of type (B3).
\begin{itemize}
\item[(1)] Assume that $\alpha$ is a homology, with center $P$ and axis $\ell$. Recall that $\cC$ is irreducible, and let $r$ be a secant line to $\cC$ passing through $P$. Then the size of $(r\cap\cH_q)\setminus\ell$ is $1$ or $2$, according to $r\cap\ell\in\cC$ or $r\cap\ell\notin\cC$, respectively.
Since $r$ and $\ell$ are fixed by $\alpha$, we have that $\alpha$ acts on the $1$ or $2$ points of $(r\cap\cH_q)\setminus\ell$. This is a contradiction to the action on the plane of $\alpha$, which has long orbits of size $3$ out of $\ell$ and $P$.
\item[(2)] Assume that $\alpha$ is of type (B1). Since $o(\alpha)=3$, this implies that $\alpha\in\PSU(3,q)$. In fact, we can use the Fermat model $X^{q+1}+Y^{q+1}+Z^{q+1}=0$ of $\cH_q$ and assume up to conjugation that $\alpha$ fixes the fundamental triangle, so that $\alpha$ is represented by a diagonal matrix $\diag(\lambda,\mu,1)$ with $\lambda^{3}=\mu^{3}=1$. As $\alpha$ is of type (B1), we have $\lambda\ne1$, $\mu\ne1$, $\mu\ne\lambda$. Then $\mu=\lambda^{-1}$ and $\det(\alpha)=1$, so that $\alpha\in\PSU(3,q)$. This contradicts $M\not\subseteq\PSU(3,q)$.
\item[(3)] Assume that $\alpha$ is of type (B3). By the Orbit-Stabilizer Theorem, the stabilizer $M_P$ in $M$ of a point $P\in I$ has order $3q(q-1)$. From $P\in\cH_q$ and \cite[Lemma 11.44]{HKT}, we have that $M_P$ is a semidirect product $M_P={M}_P^1 \rtimes {M}_P^2$, where ${M}_P^1$ is the Sylow $p$-subgroup of $M_P$ and ${M}_P^2$ is cyclic of order $3(q-1)$. Analogously, the stabilizer $H_P$ of $P$ in $H$ satisfies $H_P={H}_P^1 \rtimes {H}_P^2$ with $|{H}_P^2|=q-1$; up to conjugation, ${M}_P^2$ contains ${H}_P^2$. Since $3 \nmid (q-1)$, there exists an element $\beta \in {M}_P^2 \setminus H$ of order $3$. As $\beta$ fixes an $\mathbb{F}_{q^2}$-rational point $P$ of $\cH_q$, $\beta$ is a homology by Lemma \ref{classificazione}. Then $M=H\rtimes\langle\beta\rangle$ and a contradiction is obtained as in Case (1).
\end{itemize}
\end{proof}

\subsection{$\PSL(2,7)$ when $p=7$ or $\sqrt{-7}\notin\mathbb{F}_q$}

\begin{proposition} \label{maxpsl27}
Let $q=p^n$ be a power of an odd prime $p$ such that $3 \mid (q+1)$ and either $p=7$ or $\sqrt{-7}\notin\mathbb{F}_q$. Let $H$ be a maximal subgroup of $\PSU(3,q)$ with $H \cong \PSL(2,7)$, and $M$ be a maximal subgroup of $\PGU(3,q)$ containing $H$. Then $M \cong \PSU(3,q)$.
\end{proposition}

\begin{proof}
Assume by contradiction that $M\not\cong\PSU(3,q)$. Since $\PSU(3,q)$ is maximal in $\PGU(3,q)$ and $H$ is maximal in $\PSU(3,q)$, we have arguing as in the proof of Proposition \ref{HessianInPSU} that $M\cap\PSU(3,q)=H$ and $|M|=3|H|$.
Since $\PSU(3,q)$ is normal in $\PGU(3,q)$, $H$ is normal in $M$. Hence $M$ is a degree $3$ normal extension of $\PSL(2,7)$. From \cite[Remark 6.12]{GKeven} and \cite[Proposition 1.2 (i)]{W} we obtain that $M=H \times C_3$. Hence, $H$ acts on the points fixed by $C_3\triangleleft M$. By Lemma \ref{classificazione}, this means that $H$ is contained in one of the maximal subgroup $M_2,M_3,M_4$ described in Theorem \ref{MaximalSubgroupsPGU}, a contradiction to Theorem \ref{MaximalSubgroupsNotFixingPSU}.
\end{proof}

\subsection{The group $SmallGroup(720,765)$ when $p=5$ and $n$ is odd}

From \cite{M}, $\PSU(3,q)$ has a maximal subgroup $H$ of order $720$ which contains a normal subgroup of order $360$ isomorphic to the alternating group ${\rm A}_6$, when $p=5$ and $q$ is odd power of $p$. The following lemma gives $H$ explicitly in the GAP notation.

\begin{lemma} \label{noneraS6}
Let $H$ be a subgroup of $\PSU(3,q)$ of order $720$ containing the alternating group ${\rm A}_6$ with index $2$, where $p=5$ and $q$ is an odd power of $p$. Then $H$ is isomorphic to the subgroup of ${\rm P \Gamma L}(2,9)$ named {\rm SmallGroup(}$720$,$765${\rm )}.
\end{lemma}

\begin{proof}
By direct checking, there are exactly $4$ groups of order $720$ containing a subgroup isomorphic to ${\rm A}_6$, namely the direct product ${\rm A}_6 \times C_2$, the symmetric group ${\rm S}_6$, a semidirect product $SmallGroup(720,764)\cong {\rm A}_6\rtimes C_2$, and the $SmallGroup(720,765)$.
\begin{itemize}
\item Suppose that $M\cong{\rm A}_6\times C_2$, say $C_2=\langle\alpha\rangle$. Then by Lemma \ref{classificazione} $\alpha$ is of type (A). Since $C_2$ is normal in $H$, $H$ fixes the center of $\alpha$ and is contained in the maximal subgroup $M_2$ of $\PGU(3,q)$ described in Theorem \ref{MaximalSubgroupsPGU}. By the maximality of $H$ in $\PSU(3,q)$ we have $H=M_2\cap\PSU(3,q)$, which is a contradiction to $720=|H|\ne|M_2\cap\PSU(3,q)|=q(q+1)^2(q-1)/3$.
\item Suppose that $H\cong {\rm S}_6$. Then $H$ contains an elementary abelian subgroup of order $8$. This is a contradiction to the fact that for odd $q$ there are at most $3$ involutions which commute pairwise; see also \cite[Lemma 2.2 (viii)]{KOS}.
\item Suppose that $H$ is the semidirect product $SmallGroup(720,764)\cong {\rm A}_6\rtimes C_2$.
Then $H$ contains a cyclic subgroup $C_{10}$ of order $10$; by Lemma \ref{classificazione}, $C_{10}$ is generated by an element of type (E), and the elements of order $5$ in $C_{10}$ are of type (C).
Also, $H$ contains a dihedral subgroup $D_5$ of order $10$. If an involution normalizes an elation, then they commute, as they are both contained in a maximal subgroup of type $M_2$ and the involution is in the center of $M_2$; see \cite{MZ2} for the structure of $M_2$. Hence, the elements of order $5$ in $D_5$ are of type (D). But $H$ has a unique conjugacy class of elements of order $5$, so that we cannot have elements of type (C) and elements of type (D).
\end{itemize}
Then $H\cong SmallGroup(720,765)$ and the claim is proved.
\end{proof}

\begin{proposition}\label{max720}
Let $q=p^n$ be an odd power of $p=5$, $H$ be a maximal subgroup of $\PSU(3,q)$ isomorphic to ${\rm SmallGroup(720,765)}$, and $M$ be a maximal subgroup of $\PGU(3,q)$ containing $H$. Then $M=\PSU(3,q)$.
\end{proposition}

\begin{proof}
As in the proof of Proposition \ref{HessianInPSU}, the claim follows once that we discard the case $M\cap\PSU(3,q)=H$.
Assume by contradiction that $M\cap\PSU(3,q)=H$, which implies $|M|=3|H|=2160$.
Since $\PSU(3,q)$ is normal in $\PGU(3,q)$, also $H$ is normal in $M$. Also, ${\rm A}_6$ is characteristic in $H$ being the unique subgroup of $H$ of order $360$. Then ${\rm A_6}$ is normal in $M$.
Hence, $M/\Ker(\varphi)$ is isomorphic to an automorphism group of ${\rm A}_6$, where $\Ker(\varphi)$ is the kernel of the action by conjugation of $M$ on ${\rm A}_6$. As ${\rm A}_6\cong\PSL(2,9)$, $M/\Ker(\varphi)$ is isomorphic to a subgroup of ${\rm P\Gamma L}(2,9)$. Since the largest subgroup of ${P\Gamma L}(2,9)$ has order $720$ and the order of $M/\Ker(\varphi)$ is at least $|H|=720$, we have $|\Ker(\varphi)|=3$. Thus, $M\setminus {\rm A}_6$ contains an element $\alpha$ of order $3$ commuting with ${\rm A}_6$ elementwise.
As $3\mid(q+1)$, $\alpha$ is either of type (A), or of type (B1), or of type (B3), by Lemma \ref{classificazione}. Hence, $H$ is contained in one of the maximal subgroup $M_2,M_3,M_4$ described in Theorem \ref{MaximalSubgroupsPGU}, a contradiction to Theorem \ref{MaximalSubgroupsNotFixingPSU}.
\end{proof}

\subsection{The alternating group ${\rm A}_6$ when either $p=3$ and $n$ is even, or $5$ is a square in $\mathbb{F}_q$ but $\mathbb{F}_q$ contains no roots of unity}

\begin{proposition} \label{maxa6}
Let $q=p^n$ be an even power of $p=3$ or a power of a prime $p$ such that $5$ is a square in $\mathbb{F}_q$ but $\mathbb{F}_q$ contains no cube roots of unity. Let $H$ be a maximal subgroup of $\PSU(3,q)$ isomorphic to ${\rm A}_6$ and $M$ be a maximal subgroup of $\PGU(3,q)$ containing $H$. Then $M={\rm A}_6$ if $p=3$ and $M=\PSU(3,q)$ otherwise.
\end{proposition}

\begin{proof}
The claim is trivial for $p=3$ by Theorem \ref{MaximalSubgroupsNotFixingPSU}, as $\PGU(3,q)$ and $\PSU(3,q)$ coincide. If $p\ne3$ and $\mathbb{F}_q$ contains no cube roots of unity, then $3\mid(q+1)$ and $[\PGU(3,q):\PSU(3,q)]=3$. To prove the claim it is enough to discard the case $M\cap\PSU(3,q)=H$.
Suppose by contradiction that $M\cap\PSU(3,q)=H$, so that $|M|=3|H|$. Arguing as in the proof of Proposition \ref{max720} it can be shown that $M/\Ker(\varphi)$ is isomorphic to a subgroup of ${\rm P\Gamma L}(2,9)$, where $\varphi$ is the action by conjugation of $M$ on $H$; $H$ commutes elementwise with an element of order $3$ in $M\setminus H$; $H$ is contained either in $M_2$ or in $M_3$ or in $_4$ as described in Theorem \ref{MaximalSubgroupsPGU}, a contradiction to Theorem \ref{MaximalSubgroupsNotFixingPSU}.
\end{proof}

\subsection{The alternating group ${\rm A}_7$ when $p=5$ and $n$ is odd}

\begin{proposition}
Let $q$ be an odd power of $p=5$, $H$ be a maximal subgroup of $\PSU(3,q)$ isomorphic to ${\rm A}_7$ and $M$ be a maximal subgroup of $\PGU(3,q)$ containing $H$. Then $M=\PSU(3,q)$.
\end{proposition}

\begin{proof}
Since $3\mid(q+1)$, $\PSU(3,q)$ is normal with index $3$ in $\PGU(3,q)$, and it is enough to prove that the case $M\cap\PSU(3,q)=H$ does not occur.
If $M\cap\PSU(3,q)=H$, then as above $|M|=3|H|$ and $H$ is normal in $M$. Thus, $M\Ker(\varphi)$ is isomorphic to an automorphism group of $H$, where $\varphi$ is the action by conjugation of $M$ on $H$. Since $|\aut({\rm A}_7)|=2|{\rm A}_7|$ and $H\leq M$, we conclude that $Ker(\varphi)$ is cyclic of order $3$. Hence, $H$ commutes with an element $\alpha$ of order $3$, and $H$ acts on the points fixed by $\alpha$. This is a contradiction to Theorem\ref{MaximalSubgroupsNotFixingPSU}, according to which $H$ cannot fix a point nor a triangle.
\end{proof}

\subsection{$\PSU(3,p^m)$ where $m\mid n$ and $n/m$ is odd}

In this section $q=p^n$ can be even or odd.
Let $H$ be a subgroup of $\PSU(3,q)$ isomorphic to $\PSU(3,\bar q)$, where $\bar q=p^m$, $m$ divides $n$, and $n/m$ is odd.
We start with some lemmas on maximal subgroups of $\PSU(3,q)$ containing $H$ which we were not able to find clearly stated in \cite{M,H} nor elsewhere.
\begin{lemma}\label{lemma1sottoPSU}
Let $q=p^n$ be a prime power and $H$ be a subgroup of $\PSU(3,q)$ isomorphic to $\PSU(3,p^m)$, where $m\mid n$ and $m/n$ is odd.
Then $\PSU(3,q)$ contains a unique conjugacy class of subgroups isomorphic to $H$. Also, $H$ is not maximal in $\PSU(3,q)$ unless $n/m$ is prime.
\end{lemma}

\begin{proof}
Consider the Hermitian curve $\cH_q$ in its Norm-Trace equation $x^{q+1}=y^q+y$, let $P_{\infty}\in\cH_q$ be the unique point at infinity of $\cH_q$, and $\cA(P_{\infty})$ be the stabilizer of $P_{\infty}$ in $\PGU(3,q)$; we follow here the notation in \cite{GSX}. From \cite{GSX}, $\cA(P_{\infty})$ has order $q^3(q^2-1)$ and is a semidirect product $\mathcal{A}(P_\infty)=\mathcal{A}_1(P_\infty) \rtimes \mathcal{C}(P_\infty)$ where $\mathcal{A}_1(P_\infty)$ is the unique Sylow $p$-subgroup of $\mathcal{A}(P_\infty)$ and $ \mathcal{C}(P_\infty)$ is cyclic of order $q^2-1$.
If $\sigma \in \mathcal{A}_1(P_\infty)$, then $\sigma(x)=x+b$ and $\sigma(y)=y+b^qx+c$, for some $b,c \in \mathbb{F}_{q^2}$ satisfying $c^q+c=b^{q+1}$; if $\sigma \in \mathcal{C}(P_\infty)$, then $\sigma(x)=ax$ and $\sigma(y)=a^{q+1}y$ with $a \in \mathbb{F}_{q^2}$. Also, $\PGU(3,q)=\langle \mathcal{A}(P_\infty),w \rangle$, where $w$ is the involution defined as $w(x)=\frac{x}{y}$, $w(y)=\frac{1}{y}$.
The subgroup $\PSU(3,q)$ is generated by $\PSU(3,q)=\langle\PSU(3,q)\cap\cA(P_\infty),w\rangle$, where $\PSU(3,q)\cap\cA(P_\infty)=\mathcal{A}_1(P_\infty) \rtimes (\PSU(3,q)\cap\mathcal{C}(P_\infty))$ and $|\PSU(3,q)\cap\mathcal{C}(P_\infty)|=(q^2-1)/\gcd(3,q+1)$.
Let
$$K_1=\{\sigma \in \mathcal{A}_1(P_\infty) \mid \sigma {\rm \ is \ defined \ over} \ \mathbb{F}_{p^{2m}}\}, \quad K_2=\{\sigma \in \mathcal{C}(P_\infty) \cap \PSU(3,q)  \mid \sigma {\rm \ is \ defined \ over} \ \mathbb{F}_{p^{2m}}\},$$
and $\cA_m(P_\infty)=\langle K_1,K_2\rangle$.
Clarly, $|K_2|=p^{2m}-1$. To show that $|K_1|=p^{3m}$, let $b,c\in\mathbb{F}_{p^{2m}}$. By direct checking, the condition $c^q+c=b^{q+1}$ holds if and only if $c^{p^{m(r-2)}}+c=b^{p^{m(r-2)}+1}$, where $r=n/m$. If $r=3$, the condition reads $c^{p^m}+c=b^{p^{m(r-2)}+1}$ and has $p^{3m}$ solutions $(b,c)\in\mathbb{F}_{p^{2m}}$ by the properties of the norm and trace functions. If $r\geq5$, then by writing $m(r-2)=2m+m(r-4)$ we have that $c^{q}+c=b^{q+1}$ is equivalent to $c^{p^{m(r-4)}}+c=b^{p^{m(r-4)}+1}$; by induction, we obtain $|K_1|=p^{3m}$.
Since $K_2$ normalizes $K_1$, we have $\mathcal{A}_m(P_\infty)=K_1 \rtimes K_2$ and $|\mathcal{A}_m(P_\infty)|=p^{3m}(p^{2m}-1)$. 
Clearly, the involution $w$ is defined over $\mathbb{F}_{p^{2m}}$.
Hence, the group $\langle\cA_m(P_\infty),w\rangle$ is defined over $\mathbb{F}_{p^{2m}}$ and is isomorphic to $\PSU(3,p^m)$, since it preserves the Hermitian curve $x^{p^m+1}=y^{p^m}+y$.

If $H\leq\PSU(3,q)$ is isomorphic to $\PSU(3,p^m)$, write $H\langle S\rtimes C,\tilde{w}\rangle$, where $|S|=p^{3m}$, $C$ is cyclic of order $p^{2m}-1$, and $\tilde w$ is an involution.
The group $S\rtimes C$ fixes a point $P\in\cH_q(\mathbb{F}_{q^2})$ as $\cH_q$ has $p$-rank zero (see \cite[Theorem 11.133]{HKT}), and $\tilde w$ maps $P$ to another point $Q \in \cH_q(\mathbb{F}_{q^2})$. Since $\PSU(3,q)$ is doubly transitive on $\cH_q(\mathbb{F}_{q^2})$ and transitive on ${\rm PG}/2,q^2) \setminus \cH_q$, we have up to conjugation in $\PSU(3,q)$ that $P=P_\infty$, $Q$ is the other fixed point of $w$, and $\tilde w=w$, so that $H=\langle\cA_m(P_\infty),w\rangle$.
Hence, subgroups isomorphic to $\PSU(3,p^m)$ are conjugated in $\PSU(3,q)$ to the subgroup $\langle\cA_m(P_\infty),w\rangle$ described above.
This implies that $\PSU(3,\bar q)$ is not maximal in $\PSU(3,q)$ unless $n/m$ is prime; in fact, if $n/m=d_1 d_2$ with $d_1,d_2>1$, then there exists a proper subgroup $\PSU(3,p^{md_1})$ of $\PSU(3,q)$ which contains $\PSU(3,\bar q)$ properly.
\end{proof}

\begin{lemma}\label{lemma2sottoPSU}
Let $q=p^n$ be a prime power and $H$ be a subgroup of $\PSU(3,q)$ isomorphic to $\PSU(3,p^m)$, where $m \mid n$ and $n /m$ is an odd prime.
\begin{itemize}
\item If $n/m \ne 3$ or $3 \nmid (q+1)$, then $H$ is maximal in $\PSU(3,q)$. 
\item If $n/m=3$ and $3 \mid (q+1)$, then a maximal subgroup $M$ of $\PSU(3,q)$ containing $H$ is isomorphic to $\PGU(3,p^m)$.
\end{itemize}
\end{lemma}

\begin{proof}
As shown in the proof of Lemma \ref{lemma1sottoPSU}, we can assume up to conjugation that $H$ consists of the $\mathbb{F}_{p^{2m}}$-rational elements of $\PSU(3,q)$.
Let $M$ be a maximal subgroup of $\PSU(3,q)$ containing $H$.
Since $M$ does not leave invariant any point or triangle, we have by Theorem \ref{MaximalSubgroupsNotFixingPSU} that either $M\cong\PSU(3,p^k)$, where $k\mid n$ and $n/k$ is odd; or $M$ contains $\PSU(3,p^k)$ as a normal subgroup of index $3$, where $k\mid n$, $3\mid n/m$, and $3\mid(q+1)$.
Hence, if the odd prime $n/m$ is different from $3$ or $3\nmid(q+1)$, then the second case cannot occur, and $M\cong H$ follows.
Then we can suppose that $n/m=3$ and $3\mid(q+1)$; this implies also $3\mid(p^k+1)$, so that $\PSU(3,p^m)\ne\PGU(3,p^m)$.

Let $M$ be the subgroup of $\PGU(3,q)$ made by the elements of $\PGU(3,q)$ which are defined over $\mathbb{F}_{p^{2m}}$; note that $H\leq M$. By arguing as in the proof of Lemma \ref{lemma1sottoPSU}, we obtain that $M$ is isomorphic to $\PGU(3,p^m)$, and hence $\PSU(3,p^m)$ has index $3$ in $M$.
For any $\sigma\in M$, the determinant of $\sigma$ is an element of $\mathbb{F}_{p^{2m}}$ and hence is a cube in $\mathbb{F}_{q^2}$, because $\mathbb{F}_{q^2}$ is a cubic extension of $\mathbb{F}_{p^{2m}}$. Thus $M$ is a subgroup of $H$, and satisfies the statement of this lemma.

Let $\tilde M$ be a maximal subgroup of $\PSU(3,q)$ containing $H$ as a normal subgroup of index $3$; the claim follows if we prove that the case $\tilde M\not\cong\PGU(3,p^m)$ cannot occur.
If $\tilde M\not\cong\PGU(3,p^m)$, then from \cite[Remark 6.12]{GKeven} follows that $\tilde M$ is a semidirect product $H\rtimes C_3$. Let $C_3=\langle\alpha\rangle$; clearly $\alpha\notin M\cong\PGU(3,p^m)$, otherwise $\tilde M$ and $M$ coincide.
By Lemma \ref{classificazione}, $\alpha$ is either of type (A), or of type (B1), or of type (B3).
\begin{itemize}
\item Suppose that $\alpha$ is of type (A) or (B1). As in Case (2) in the proof of Proposition \ref{maxPGLconica}, we can use the Fermat model of $\cH_q$ and assume up to conjugation that $\alpha$ is a diagonal matrix of type $\diag(\lambda,\mu,1)$, where $\lambda^3=\mu^3=1$; here, $\lambda=\mu$ or $\lambda\ne\mu$ according to $\alpha$ being of type (A) or (B1), respectively. In any case, $\alpha$ is defined over $\mathbb{F}_{p^{2m}}$ as $3\mid(p^m+1)$, and hence $\alpha\in M\cong\PGU(3,p^m)$, a contradiction.
\item Suppose that $\alpha$ is of type (B3), and let $T$ be the $\mathbb{F}_{q^6}$-rational triangle fixed pointwise by $\alpha$. Then $\alpha$ is an element of the pointwise stabilizer of $T$ in $\PSU(3,q)$, which has order $(q^2-q+1)/3$. This is a contradiction to $o(\alpha)=3\nmid(q^2-q+1)/3$.
\end{itemize}
\end{proof}

\begin{proposition} \label{maxpgupsu}
Let $q=p^n$ be a prime power and $H$ be a subgroup of $\PSU(3,q)$ isomorphic to $\PSU(3,p^m)$, where $m \mid n$ and $n /m$ is an odd prime. Let $M$ be a maximal subgroup of $\PSU(3,q)$ containing $H$, so that
\begin{itemize}
\item if $n/m\ne3$ or $3\nmid(q+1)$, then $M\cong\PSU(3,p^m)$;
\item if $n/m=3$ and $3\mid(q+1)$, then $M\cong\PGU(3,p^m)$.
\end{itemize}
Let $\tilde M$ be a maximal subgroup of $\PGU(3,q)$ containing $M$.
\begin{itemize}
\item If $n/m \ne 3$ or $3 \nmid (q+1)$, then $\tilde M \cong \PGU(3,p^m)$.
\item If $n/m=3$ and $3 \mid (q+1)$, then $\tilde M \cong \PSU(3,q)$.
\end{itemize}
\end{proposition}

\begin{proof}
Assume without restriction that $\cH_q$ has Norm-Trace equation $x^{q+1}=y^q+y$ and, up to conjugation, the elements of $H$ are defined over $\mathbb{F}_{p^{2m}}$.
The claim on $M$ follows from Lemma \ref{lemma2sottoPSU}.
If $3 \nmid (q+1)$, then $\PSU(3,q)=\PGU(3,q)$, and hence $\tilde M=M \cong \PGU(3,p^m)$ by Lemma \ref{lemma2sottoPSU}.

Now suppose that $n/m \ne 3$ and $3 \mid (q+1)$. Then either $\tilde M \cap \PSU(3,q)=M$, or $\tilde M \cap \PSU(3,q)=\PSU(3,q)$; in the latter case $\tilde M=\PSU(3,q)$.
We prove the claim by providing a maximal subgroup $\hat M$ of $\PGU(3,q)$ isomorphic to $\PGU(3,p^m)$ and showing that the case $\tilde M\not\cong\hat M$ cannot occur. Arguing as in the proof of Lemma \ref{lemma1sottoPSU}, let $\hat M$ be the subgroup of $\PGU(3,q)$ made by the elements which are defined over $\mathbb{F}_{p^{2m}}$. Then $\hat M\cong\PGU(3,q)$. Since $\hat M$ and $\PSU(3,q)$ are maximal in $\PGU(3,q)$, we have $\hat M\cap\PSU(3,q)=M\cong\PSU(3,p^m)$ and $[\hat M : M]=3$.
Let $N$ be a maximal subgroup of $\PGU(3,q)$ such that $\hat M\leq N$. Then $N\cap\PSU(3,q)\ne\PSU(3,q)$ by the maximality of $\PSU(3,q)$ in $\PGU(3,q)$; hence, $N\cap\PSU(3,q)=M$ by the maximality of $M$ in $\PSU(3,q)$, and $[N:M]=3$. Thus $|N|=|M|$ and $N=M$, proving that $\hat M$ is maximal in $\PGU(3,q)$.
Assume by contradiction that $\tilde M\not\cong\hat M$. Then by \cite[Remark 6.12]{GKeven} $\tilde M$ is a semidirect product $H\times C_3$, where $C_3=\langle\alpha\rangle$ and $\alpha\notin\hat{M}$. By Lemma \ref{classificazione}, $\alpha$ is either of type (A), or of type (B1), or of type (B3). If $\alpha$ is of type (A) or (B1), a contradiction follows as in the proof of Lemma \ref{lemma2sottoPSU}. If $\alpha$ is of type (B3) and $T$ is the $\mathbb{F}_{q^6}$-rational triangle fixed by $\alpha$, then $\alpha$ is the unique element of order $3$ in the pointwise stabilizer$C$ of $T$ in $\PGU(3,q)$, which has order $q^2-q+1$.
Since $n/m\ne3$, it is easily seen that $(p^{2m}-p^m+1)\mid(q^2-q+1)$. Thus, the intersection $C\cap\hat{M}\cong\PGU(3,p^m)$ has order $p^{2m}-p^m+1$, which is divisible by $3$; hence $\alpha\in\hat M$, a contradiction.

Finally, suppose that $n/m=3$ and $3\mid(q+1)$.
Assume by contradiction that $\tilde M\ne\PSU(3,q)$, so that $\tilde M \cap \PSU(3,q)=M$ and $|\tilde M|=3|M|=9|\PSU(3,p^m)|$. Since $M$ is normal in $\tilde M$, we have from \cite[Remark 6.12]{GKeven} that $\tilde M$ is a semidirect product $\tilde M \cong \PGU(3,p^m) \rtimes C_3$ and $3 \mid m$. Let $C_3=\langle \alpha \rangle$, so that by Lemma \ref{classificazione} $\alpha$ is either of type (A), or of type (B1), or of type (B3). Arguing as above, $\alpha \notin \PSU(3,q)$ implies that $\alpha$ is of type (B3).
Since the elements of type (A) cannot stabilize the $\mathbb{F}_{q^6}$-rational triangle fixed poitwise by $\alpha$, we have that $\alpha$ does not commute elementwise with $\PGU(3,p^m)$.
Hence, $\tilde M$ is a subgroup of the automorphism group ${\rm P\Gamma U}(3,p^m)=\PGU(3,p^m)\rtimes C_m$ of $\PGU(3,p^m)$; here, $C_m$ has order $m$ and is generated by the Frobenius automorphism $\varphi:(X,Y,Z)\mapsto(X^p,Y^p,Z^p)$. Up to multiplying by an element of $\PGU(3,p^m)$, we can assume that $\alpha\in C_m$. Thus, $\alpha$ commutes those elements of $\PGU(3,p^m)$ which are defined over $\mathbb{F}_{p}$; in particular, $\alpha$ commutes with the elements $\sigma:(x,y)\mapsto(ax,a^{q+1}y)$, $a\in\mathbb{F}_p$, which are of type (B2). This implies that $\alpha$ fixes a point $P\in {\rm PG}(2,q^2)\setminus\cH_q$, a contradiction to $\alpha$ being of type (B3).
\end{proof}

\subsection{The group $SmallGroup(36,9)$ when $q=2$}

From \cite{H}, $\PSU(3,2)$ has a maximal subgroup $H$ of order $36$. By direct checking, $H\cong SmallGroup(36,9)$ in the GAP notation.

\begin{proposition}\label{36ndispari}
Let $q=2^n$ be a power of $2$. Then $\PSU(3,q)$ contains a subgroup isomorphic to $SmallGroup(36,9)$ if and only if $n$ is odd.
\end{proposition}

\begin{proof}
If $n$ is odd, the claim follows from Theorem \ref{MaximalSubgroupsNotFixingPSU} as $PSU(3,q)$ contains $\PSU(3,2)$.
Conversely, suppose by contradiction that $SmallGroup(36,9)\leq\PSU(3,q)$ with $n$ even. Then $\PSU(3,q)$ contains an elementary abelian $3$-subgroup $E$ of order $9$, whose nontrivial elements are of type (B2) from Lemma \ref{classificazione}.
Let $\sigma,\tau\in E\setminus\{id\}$ with $\tau\notin\langle\sigma\rangle$; then $\tau$ acts on the fixed points $\{P_1,P_2,P_3\}$ of $\sigma$, where $P_1\in\PG(2,q^2)\setminus\cH_q$ and $P_2,P_3\in\cH_q(\mathbb{F}_{q^2})$. By $E\leq\aut(\cH_q)$ and the Orbit-Stabilizer Theorem, $E$ fixes $\{P_1,P_2,P_3\}$ pointwise.
Hence, $E$ is cyclic by \cite[Lemma 11.44 (c)]{HKT}, a contradiction.
\end{proof}

\begin{remark}\label{rimarco2}
Let $q$ be an odd power of $2$. Then the only subgroups of $\PGU(3,q)$ properly containing $SmallGroup(36,9)$ are $\PSU(3,2^k)$ and $\PGU(3,2^k)$, where $k\mid n$.

In fact, by direct checking, any degree $3$ extension $M$ of $SmallGroup(36,9)$ contains a normal subgroup of order $3$; if $M\leq\PGU(3,q)$ then $M$ would fix a point or a triangle and the same does $SmallGroup(36,9)$, a contradiction.
Then the claim follows from Theorem {\rm \ref{MaximalSubgroupsNotFixingPSU}}.
\end{remark}

\begin{remark}\label{rimarco3}
When $q$ is even, the only subgroups $G$ of $\PSU(3,q)$ which do not fix any point or triangle are $G\cong\PSU(3,2^k)$ when $k\mid n$, $n$ is odd, and $3$ divides $n/k$; $G\cong\PGU(3,2^k)$ when $k\mid n$ ; and $G\cong SmallGroup(36,9)$ when $n$ is even.

In fact, from Theorem \ref{MaximalSubgroupsNotFixingPSU}, either $q=2$ and $G\leq SmallGroup(36,9)$; or $G\leq\PSU(3,2^k)$; or $G\leq\PGU(3,2^k)$.
If $G< SmallGroup(36,9)$, then $G$ contains a cyclic normal subgroup, and hence $G$ fixes a point or a triangle.
If $G<\PSU(3,2^k)$ or $G<\PGU(3,2^k)$, apply inductively the same argument.
\end{remark}

\section{Genera of $\cH_q/G$ when $G\leq\PGU(3,q)$ has no fixed points or triangles in $\PG(2,\bar{\mathbb{F}}_{q^2})$}\label{sec:generi}

In this section we compute the genera of all quotients $\cH_q/G$ such that $G\leq\PGU(3,q)$ has no fixed points (hence also no fixed lines) nor fixed triangles in the plane $\PG(2,\bar{\mathbb{F}}_{q^2})$. This is equivalent to require that $G$ has no fixed points or triangles in $\PG(2,q^6)$ (see Lemma \ref{classificazione}).
We proceed with a case-by-case analysis on maximal subgroups $M$ of $\PSU(3,q)$ and $\PGU(3,q)$ containing $G$, as described in Theorem \ref{MaximalSubgroupsNotFixingPSU} and Section \ref{sec:sottogruppimassimaliPGU}.
In order to avoid the subgroups $G\leq\PGU(3,q)$ which fix a point or a triangle, the following remark will be useful.

\begin{remark}\label{rimarco}
If $G\leq\PGU(3,q)$ has a cyclic normal subgroup $C$, then $G$ fixes either a point or a triangle in $\PG(2,q^6)$. In fact, $G$ acts on the points fixed by a generator $\alpha$ of $C$; hence, according to Lemma {\rm \ref{classificazione}}, $G$ fixes either a point, if the type of $\alpha$ is {\rm (A)} or {\rm (B2)} or {\rm (C)} or {\rm (D)} or {\rm (E)}; or a triangle, if the type of $\alpha$ is {\rm (B1)} or {\rm (B3)}.
\end{remark}

\subsection{$G$ is a subgroup of the Hessian group $H_{216}\cong\PGU(3,2)$}

\begin{proposition}\label{propohessian}
Let $q=p^n$ be a power of an odd prime $p$ with $3\mid(q+1)$, and $H$ be a subgroup of $\PGU(3,q)$ isomorphic to the Hessian group $H_{216}\cong\PGU(3,2)$.
If $G$ is a subgroup of $H$ such that $G$ has no fixed points or triangles in $\PG(2,q^6)$, then the genus of $\cH_q/G$ is one of the following:
\begin{equation}\label{hessian1}
\frac{q^2-34q+397-54\delta}{432},\quad \frac{q^2-10q+133-54\delta}{144},\quad \frac{q^2-10q+61-18\delta}{72},\quad \textrm{where}\quad \delta=\begin{cases} 0, & \textrm{if}\quad 4\mid(q+1), \\ 2, & \textrm{if}\quad 4\mid(q-1), \end{cases}
\end{equation}
and $G\cong\PGU(3,2)$, $G\cong\PSU(3,2)$, $G\cong SmallGroup(36,9)$, respectively.
Conversely, if $\bar g$ is one of the integers in Equation \eqref{hessian1}, then there exists a subgroup $G$ of $H$ such that $G$ has no fixed points or triangles in $\PG(2,q^6)$ and $g(\cH_q/G)=\bar g$.
\end{proposition}

\begin{proof}
By direct checking, $H\cong\PGU(3,2)$ has just $3$ conjugacy classes of subgroups which do not admit cyclic normal subgroups; namely, they are $H$ itself, $K\cong\PSU(3,2)$ (which is normal in $H$), and three groups $L_1,L_2,L_3\cong SmallGroup(36,9)$ of order $36$ (which are normal in $K$).
By Remark \ref{rimarco}, we only have to consider the cases $G\in\{H,K,L_1\}$.
To prove the claim, we will show that the first, second, and third value in Equation \eqref{hessian1} are the genus of $\cH_q/G$ with $G=H$, $G=K$, and $G=L_1$, respectively.

{\bf Case $G=H\cong\PGU(3,2)$.} The group $G$ contains
\begin{itemize}
\item a conjugacy class $S_1$ consisting of $9$ elements of order $2$;
\item a conjugacy class $S_2$ consisting of $8$ elements of order $3$;
\item a conjugacy class $G_3$ consisting of $12$ subgroups of order $3$; $G_3$ splits into $2$ conjugacy class $S_3$, $S_3^{\prime}$, each one consisting of $12$ elements of order $3$;
\item a conjugacy class $G_4$ consisting of $24$ subgroups of order $3$; $G_4$ splits into $2$ conjugacy class $S_4$, $S_4^{\prime}$, each one consisting of $24$ elements of order $3$;
\item a conjugacy class $S_5$ consisting of $54$ elements of order $4$;
\item a conjugacy class $G_6$ consisting of $36$ cyclic subgroups of order $6$; $G_6$ splits into $2$ conjugacy class $S_6$, $S_6^{\prime}$, each one consisting of $36$ elements of order $6$.
\end{itemize}
We use Lemma \ref{classificazione} to determine the type of nontrivial elements $\sigma\in G$.
If $\sigma\in S_1$, then $\sigma$ is of type (A).

Let $\sigma\in S_5$. If $4\mid(q-1)$, then $\sigma$ is of type (B2). If $4\mid(q+1)$, then $\sigma$ is of type (B1). In fact, if $\sigma$ is not of type (B1), then $\sigma$ is of type (A) and $S_5$ contains only another element $\tau$ different from $\sigma$ such that $\tau^2=\sigma^2$, namely $\tau=\sigma^{-1}$, because the homologies with given center form a cyclic group; but the number $|S_1|=9$ of involutions in $G$ is strictly smaller than $|S_5|$, a contradiction.

Let $\sigma\in S_6$. Recall that $6\mid(q+1)$. Since a cyclic group of order $6$ contains exactly $2$ elements of order $6$ and $|S_6|/2 >|C_1|$, the same argument used for $S_5$ shows that $\sigma$ is of type (B1).

Let $\sigma\in S_2$.
By direct checking, $\sigma$ is contained in a subgroup $S_3=\langle\sigma\rangle\rtimes C_2$ of $H$. Hence, an involution acts on the fixed points of $\sigma$, so that $\sigma$ is not of type (B3), because an element of type (A) cannot fix a point which is not $\mathbb{F}_{q^2}$-rational.
The elements of $S_2$, together with the identity, form an elementary abelian group $C_3\times C_3$. Being conjugated, the elements of $S_2$ are all of the same type. If $C_3\times C_3$ is generated by elements of type (A), then $C_3\times C_3$ contains $2$ elements of type (B1); in fact, using the Fermat model of $\cH_q$ we have up to conjugation in $\PGU(3,q)$ that $C_3\times C_3=\{\diag(\lambda,\mu,1)\mid\lambda^3=\mu^3=1\}$, and $\diag(\rho,\rho^{-1},1)$, $\diag(\rho^{-1},\rho,1)$ are of type (B1), where $\rho$ is a primitive cube root of unity.
Therefore, the elements of $S_2$ are of type (B1).

Let $\sigma\in S_3$ or $\sigma\in S_3^{\prime}$; since elements of $S_3$ and elements of $S_3^{\prime}$ are inverse each other, they are of the same type. As $\langle\sigma\rangle$ is normalized by an involution, the same argument used for $S_2$ shows that $\sigma$ is not of type (B3).
We show that $\sigma$ is of type (A); to this aim, assume by contradiction that the elements of $S_3$ and $S_3^{\prime}$ are of type (B1).
By direct checking, $\sigma$ is contained in an elementary abelian $3$ group $E=C_3\times C_3$; $E$ contains a subgroup $\bar{C}_3$ of order $3$ whose nontrivial elements are in $S_2$, while every element in $E\setminus\bar{C}_3$ is in $S_3$ or $S_3^{\prime}$.
Let $\sigma^{\prime}\in E\setminus\langle\sigma\rangle$, and denote by $T=\{P_1,P_2,P_3\}$ and $T^{\prime}=\{P_1^{\prime},P_2^{\prime},P_3^{\prime}\}$ the triangles fixed pointwise by $\sigma$ and $\sigma^{\prime}$, respectively.
Then $T\ne T^{\prime}$, since there are just $3$ elements of order $3$ and type (B1) which fix $T$ pointwise.
Also, $T$ and $T^{\prime}$ have no vertex in common. In fact, $\sigma$ acts on $T^{\prime}$, because $\sigma$ and $\sigma^{\prime}$ commute; if $T\ne T^{\prime}$ and $P_1=P_1^{\prime}$ (as represented in the following picture), then $\sigma$ acts on $\{P_2^{\prime},P_3^{\prime}\}$ with long orbits, a contradiction to $o(\sigma)=3$.

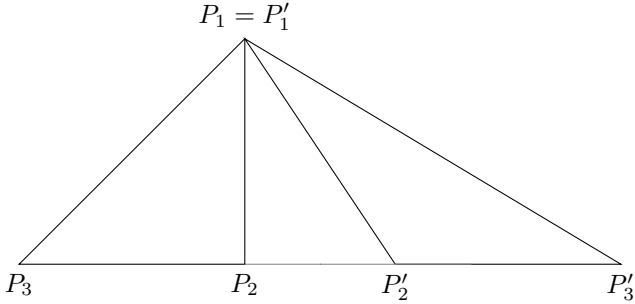
\begin{tikzpicture}
\draw (0,0) node[anchor=north]{$P_3$}
  -- (3,0) node[anchor=north]{$P_2$}
  -- (3,3) node[anchor=south]{$P_1 = P_1^\prime$}
  -- cycle;
\draw (8,0) node[anchor=north]{$P_3^\prime$}
  -- (5,0) node[anchor=north]{$P_2^\prime$}
  -- (3,3) node[anchor=south]{}
  -- cycle;
 \draw[very thin,color=gray] (3,0) grid (6,0);
\end{tikzpicture}

Therefore, the vertices of the triangles $T_1,T_2,T_3,T_4$ fixed pointwise by some nontrivial element of $E$ are $12$ distinct points altogether.
By direct checking, the normalizer of $E$ in $G$ contains a subgroup $N\cong E\rtimes C_2$ of order $18$.
Let $C_2=\langle\tau\rangle$. Since $\tau$ normalizes $E$, $\tau$ acts on $T_1,T_2,T_3,T_4$.
By direct checking, $N$ contains elements of order $6$. Hence, $\tau$ commutes with some nontrivial $\alpha\in E$, and thus $\tau$ fixes pointwise a triangle $T_i$, say $T_1=\{P,Q,R\}$; we can assume that the center of the homology $\tau$ is $P$.
Then the involution $\tau$ acts on the remaining $9$ vertices, and thus fixes another point $S$, vertex of a triangle $\{S,U,V\}$.
Since $\tau$ fixes $S$ and $S$ cannot be the center of $\tau$, $S$ is a point of the axis $QR$ of $\tau$, as represented in the following picture.

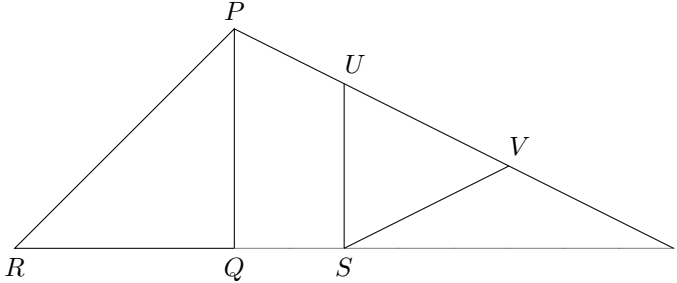
\begin{tikzpicture}[scale=0.73]
\draw (0,0) node[anchor=north]{$R$}
  -- (4,0) node[anchor=north]{$Q$}
  -- (4,4) node[anchor=south]{$P$}
  -- cycle;
 \draw[very thin,color=gray] (4,0) grid (12,0);
 \draw[color=black] (4,4) -- (12,0);
\draw (6,0) node[anchor=north]{$S$};
 \draw[thin] (6,0) -- (6,3);
 \draw[thin, color=black] (6,0) -- (9,1.5);
\draw (6.2,3.7) node[anchor=north]{$U$};
\draw (9.2,2.2) node[anchor=north]{$V$};
\end{tikzpicture}

Since $\alpha$ acts on the line $QR$ and $\alpha$ acts also on the triangle $\{S,U,V\}$, $\alpha$ should fix the point $S$, a contradiction. We have then shown that the elements in $S_3$ and $S_3^{\prime}$ are of type (A).

Let $\sigma\in S_4$ or $\sigma\in S_4^{\prime}$; the elements of $S_4$ and $S_4^{\prime}$ are of the same type as they are inverse each other.
Let $\gamma\in\{0,3,q+1\}$ be such that $i(\sigma)=\gamma$.
Then, by the Riemann-Hurwitz formula and Theorem \ref{caratteri},
$$q^2-q-2=216\left(2g(\cH_q/G)-2\right)+9(q+1)+24(q+1)+48 \gamma + 54 \delta,$$
where $\delta=0$ if $4\mid(q+1)$, and $\delta=2$ if $4\mid(q-1)$.
As $g(\cH_q/G)$ is an integer, we have $48\gamma\equiv0\pmod{27}$, and hence $\gamma\in\{0,q+1\}$. We show that $\gamma=0$.
By direct checking, $H$ contains $8$ elementary abelian $3$-subgroups $E_1,\ldots,E_8$ of order $9$; for any $i=1,\ldots,8$, $L_i$ contains $2$ elements of $S_2$ and $6$ elements of $S_4\cup S_4^{\prime}$; for any $\tau\in S_2$, there exist exactly $2$ indexes $i,j\in\{1,\ldots,8\}$ such that $\langle\tau\rangle\subset E_i$ and $\langle\tau\rangle\subset E_j$.
Assume by contradiction that $\gamma=q+1$. Hence, the elements of $E_i\setminus\langle\tau\rangle$ and $E_j\setminus\langle\tau\rangle$ are of type (A). Then $E_i$ and $E_j$ are generated by elements of type (A), so that they fix pointwise two triangles $T_i$ and $T_j$. Thus, $\tau$ is an element of type (B1) fixing both $T_i$ and $T_j$ pointwise; this implies $T_i=T_j$. This yields the contradiction $E_i=E_j$, because the pointwise stabilizer of $T_i$ has the form $C_{q+1}\times C_{q+1}$ which has a unique elementary abelian $3$-subgroup of order $9$.

To sum up, the elements of $S_1$, $S_3$, and $S_3^{\prime}$ are of type (A); the elements of $S_2$, $S_4$, $S_4^{\prime}$, $S_6$, and $S_6^{\prime}$ are of type (B1); the elements of $S_5$ are of type (B1) or (B2) according to $4\mid(q+1)$ and $4\mid(q-1)$, respectively.
By the Riemann-Hurwitz formula and Theorem \ref{caratteri},
$$ g(\cH_q/G)=\frac{q^2-34q+397-54\delta}{432},\qquad\textrm{where}\quad \delta=\begin{cases} 0, & \textrm{if}\quad 4\mid(q+1), \\ 2, & \textrm{if}\quad 4\mid(q-1). \end{cases} $$

{\bf Case $G=K\cong\PSU(3,2)$.} The group $G$ contains $9$ elements of order $2$, $54$ elements of order $4$, and $8$ elements of order $3$ contained in $S_2$. By the Riemann-Hurwitz formula,
$$ q^2-q-2=72\left(2g(\cH_q/G)-2\right)+9(q+1)+54\delta+8\cdot0, $$
and hence
$$ g(\cH_q/G)=\frac{q^2-10q+133-54\delta}{144}. $$

{\bf Case $G=L_1\cong SmallGroup(36,9)$.} The group $G$ contains $9$ elements of order $2$, $18$ elements of order $4$, and $8$ elements of order $3$ contained in $S_2$. By the Riemann-Hurwitz formula,
$$ g(\cH_q/G)=\frac{q^2-10q+61-18\delta}{72}. $$

Finally, we note that $L_1$ does not fix any point or triangle, and hence the same holds for $K$ and $H$ which contain $L_1$.
In fact, the group $L_1$ cannot fix any point $P\in\cH_q$, since $L_1$ contains elements of type (B1). The group $L_1$ cannot fix any point $P\in\PG(2,q^2)\setminus\cH_q$, since $L_1$ contain an elementary abelian $3$-subgroup $C_3\times C_3$ whose nontrivial elements are of type (B1); so that if $\sigma_1,\sigma_2\in (C_3\times C_3)\setminus\{id\}$ and $\sigma_2\notin\langle\sigma_1\rangle$, then $\sigma_2$ acts without fixed points on the $3$ points fixed by $\sigma_1$.
The group $L_1$ cannot fix any triangle $T\subset\cH_q(\mathbb{F}_{q^6})$, since the stabilizer of $T$ in $\PGU(3,q)$ has odd order $3(q^2-q+1)$ unlike $L_1$.
The group $L_1$ cannot fix any triangle $T\subset\PG(2,q^2)\setminus\cH_q$, since the elements of order $4$ should fix a vertex of $T$ and interchange the other two vertexes of $T$; so that their squares, that is the $9$ involutions of $L_1$, should fix $T$ pointwise, a contradiction to the fact that there are exactly $3$ involutions in $\PGU(3,q)$ fixing $T$ pointwise.
\end{proof}

\subsection{$G$ is a subgroup of $\PGL(2,q)$ preserving a conic}

Recall that $\PGU(3,q)$ contains a subgroup $H\cong\PGL(2,q)$ preserving a conic $\cC$, namely a Baer conic, i.e. the restriction of an irreducible conic to a Baer subplane of $\PG(2,q^2)$; see \cite{CE} for a description of $\cC$ and $H$.

\begin{proposition}\label{propopgl}
Let $q=p^n$ be a power of an odd prime $p$, and $H\cong\PGL(2,q)$ be a subgroup of $\PGU(3,q)$ preserving an irreducible conic $\cC$.
If $G$ is a subgroup of $H$ such that $G$ has no fixed points or triangles in $\PG(2,q^6)$, then the genus of $\cH_q/G$ is one of the following:
\begin{equation}\label{pgl1}
\frac{q^2-16q+103-24\gamma-20\delta}{120},
\end{equation}
when $p=5$ or $5\mid(q^2-1)$, with $G\cong{\rm A}_5$ and
$$\delta=\begin{cases} 2, \ if \ either \ p=3 \ or \ 3 \mid (q-1), \\ 0, \ if \ 3 \mid (q+1),\end{cases} \quad and \quad \gamma=\begin{cases} 0, \ if  \ 5 \mid (q+1), \\ 2, \ if \ p=5 \ or \ 5 \mid (q-1); \end{cases}$$
\begin{equation}\label{pgl2}
\frac{q^2-q-2-\Delta}{\bar q (\bar q +1)(\bar q -1)}+1,
\end{equation}
where $q=\bar{q}^h$, $\bar q\ne3$, $G\cong\PSL(2,\bar q)$, and
\begin{itemize}
\item $\Delta=+2(\bar q -2)( \bar q+1) +2 \frac{\bar q(\bar q+1)}{2} \bigg( \frac{\bar q-1}{2}-2\bigg)+\frac{\bar q ( \bar q+1)}{2}(q+1)+ \delta \frac{\bar q(\bar q-1)}{2} \bigg( \frac{\bar q+1}{2}-1\bigg),$ if $\bar q \equiv 1 \pmod 4$,
\item $\Delta=+ 2(\bar q -1)( \bar q+1) + 2\frac{\bar q(\bar q+1)}{2} \bigg( \frac{\bar q-1}{2}-1\bigg)+\frac{\bar q ( \bar q+1)}{2}(q+1)+\delta \frac{\bar q(\bar q-1)}{2} \bigg( \frac{\bar q+1}{2}-2\bigg)$, if $\bar q \equiv 3 \pmod 4$,
$${\rm with} \ \delta=\begin{cases} 2, \ if \ h \ is \ even, \\ 0, \ otherwise ;\end{cases}$$
\end{itemize}
\begin{equation}\label{pgl3}
\frac{q^2-q-2-\Delta}{2\bar q (\bar q +1)(\bar q -1)}+1,
\end{equation}
where $q=\bar{q}^h$, $\bar q\ne3$, $G\cong\PGL(2,\bar q)$, and
$$\Delta=2(\bar q-1)(\bar q+1)+\frac{\bar q ( \bar q+1)}{2}(q+1)+\frac{\bar q ( \bar q-1)}{2}(q+1)+2\frac{\bar q ( \bar q+1)}{2}(\bar q-1-2)+\delta \frac{\bar q ( \bar q-1)}{2}(\bar q+1-2)$$
and
$$\delta=\begin{cases} 2, \ if \ h \ is \ even, \\ 0, \ otherwise .\end{cases}$$
Conversely, if $\bar g$ is one of the integers in Equations \eqref{pgl1} to \eqref{pgl3}, then there exists a subgroup $G$ of $H$ such that $G$ has no fixed points or triangles in $\PG(2,q^6)$ and $g(\cH_q/G)=\bar g$.
\end{proposition}

\begin{proof}
If $G$ has no fixed points or triangles, then the following holds: $G$ has no cyclic normal subgroups, by Remark \ref{rimarco}; $G$ has no elementary abelian normal $2$-subgroups of order $4$; $G$ has no normal $p$-subgroups.
Then by \cite[Hauptsatz 8.27]{Hup} $G=K_1\cong {\rm A}_5$ with $p=5$ or $5\mid(q^2-1)$, or $G=K_2\cong\PSL(2,\bar q)$, or $G=K_3\cong\PGL(2,\bar q)$, with $q=\bar{q}^h$.

Elements $\alpha$ of order $p$ in $H$ are of type (D)
In fact, suppose by contradiction that $\alpha$ is an elation. Then $\alpha$ fixes a point $P \in \cH_q \cap \cC$, where $\cC$ is the fixed conic, and its tangent line pointwise; also, it acts with long orbits on the remaining $q$ points belonging to each other line containing $P$. Since every of these lines intersects the conic $\cC$ in either $1$ or $2$ points, and this intersection is preserved by $\sigma$, each line must be tangent to $\cC$ at $P$, a contradiction.

Elements $\alpha$ of order dividing $q+1$ and different from $2$ are of type (B1).
In fact, $H$ contains a dihedral group $D=\langle \alpha \rangle \rtimes C_2$ and hence $\langle \alpha \rangle$ is normalized by an involution. If $\alpha$ is a homology, then $\alpha$ and the involution commute, a contradiction. Also, no elements of type (B3) are normalized by an involution.

Suppose that $G=K_1$. By the Riemann-Hurwitz formula and Theorem \ref{caratteri}, the genus of $\cH_q/G$ is given by Equation \eqref{pgl1}.
Suppose that $G=K_2$ or $G=K_3$. The order statistics of $G$ follows from the analysis of $\PSL(2,\bar q)$ and $\PGL(2,\bar q)$ in \cite[Chapter II.8]{Hup}. Together with the Riemann-Hurwitz formula and Theorem \ref{caratteri}, Equations \eqref{pgl2} and \eqref{pgl3} provide the genus of $\cH_q/K_2$ and $\cH_q/K_3$, respectively.

Finally, we show that $K_1$, $K_2$, and $K_3$ do not fix any point or triangle.
The group $K_1$ contains $15$ involutions which form $5$ elementary abelian $2$-groups $E_1,\ldots,E_5$ of order $4$, that intersect pairwise trivially.
Hence, the triangles fixes pointwise by $E_i$ and $E_j$ are disjoint for $i\ne j$. This implies that $K_1$ cannot fix any point $P\in\cH_q$ nor any self-polar triangle $T\subset\PG(2,q^2)\setminus\cH_q$; also, $K_1$ cannot fix any triangle $T\subset\cH_q(\mathbb{F}_{q^6})$ as $K_1$ has even order.
The groups $K_2$ and $K_3$ contain $p$-elements.
The groups $K_2$ and $K_3$ cannot fix a point $P\in\cH_q(\mathbb{F}_{q^2})$, since $K_2$ and $K_3$ contain elementary abelian $2$-subgroups of order $4$.
The groups $K_2$ and $K_3$ cannot fix a point $P\in\PG(2,q^2)\setminus\cH_q(\mathbb{F}_{q^2})$ as they contain $p$-elements of type (D).
The groups $K_2$ and $K_3$ cannot fix cannot fix a triangle $T\subset\cH_q(\mathbb{F}_{q^6})$ as they have even order.
The groups $K_2$ and $K_3$ cannot fix a triangle $T\subset\PG(2,q^2)\setminus\cH_q$. Otherwise, they contain an abelian subgroup of index dividing $6$, namely the pointwise stabilizer of $T$; since an abelian subgroup has order at most $q+1$, this yields $\bar q=3$.
\end{proof}

\subsection{$G$ is a subgroup of $\PSL(2,7)$ with $p=7$ or $\sqrt{-7}\notin\mathbb{F}_q$}

\begin{proposition} \label{psl27}
Let $q=p^n$ be a power of an odd prime $p$, where either $p=7$ or $\sqrt{-7}\notin\mathbb{F}_q$, and $H \leq \PGU(3,q)$ be isomorphic to $\PSL(2,7)$.
The the genus of the quotient curve $\cH_q/H$ is
$$\frac{q^2-22q+313-56\alpha-48\beta -42\gamma}{336},$$
where 
$$\alpha=\begin{cases} 0, \ if \ 3 \mid (q+1), \\ 2, \ otherwise; \end{cases} \beta=\begin{cases} 0, \ if \ 7 \mid (q+1), \\ 3, \ if \ 7 \mid (q^2-q+1), \\ 2, \ otherwise; \end{cases} \ and \ \gamma=\begin{cases} 0, \ if \ 4 \mid (q+1), \\ 2, \ otherwise. \end{cases}$$
Also, every proper subgroup of $H$ fixes a point or a triangle.
\end{proposition}

\begin{proof}
The group $H$ contains $21$ elements of order $2$, $56$ elements of order $3$, $48$ elements of order $7$, and $42$ elements of order $4$.

\textbf{Case $p=7$.} We prove that elements of order $7$ are of type (D); note that all elements $\sigma\in H$ of order $7$ are of the same type, as they form a unique conjugacy class in $H$.
Assume by contradiction that $\sigma$ is of type (C).
Since $\sigma$ does not commute with any involution, $\sigma$ acts with long orbits of length $7$ on the set of the $21$ centers $P_1,\ldots,P_{21}$ of the involutions of $H$.
Hence, $P_1,\ldots,P_{21}$ are on three lines $\ell_1,\ell_2,\ell_3$ containing the center $P$ of $\sigma$, as represented in the following picture.

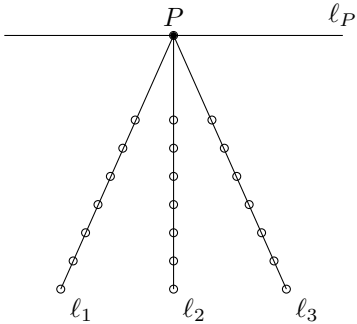
\begin{tikzpicture}[scale=0.75]
 \draw (2,4.5) -- (8,4.5);
 \draw (5,4.5) -- (5,0);
 \draw (5,4.5) -- (3,0);
 \draw (5,4.5) -- (7,0);
\draw (5,4.5) circle[radius=2pt];
\draw (5,3) circle[radius=2pt];
\draw (5,2.5) circle[radius=2pt];
\draw (5,2) circle[radius=2pt];
\draw (5,1.5) circle[radius=2pt];
\draw (5,1) circle[radius=2pt];
\draw (5,0.5) circle[radius=2pt];
\draw (5,0) circle[radius=2pt];  

\draw (4.32,3) circle[radius=2pt];
\draw (4.1,2.5) circle[radius=2pt];
\draw (3.88,2) circle[radius=2pt];
\draw (3.66,1.5) circle[radius=2pt];
\draw (3.44,1) circle[radius=2pt];
\draw (3.22,0.5) circle[radius=2pt];
\draw (3,0) circle[radius=2pt];  

\draw (5.68,3) circle[radius=2pt];
\draw (5.9,2.5) circle[radius=2pt];
\draw (6.12,2) circle[radius=2pt];
\draw (6.34,1.5) circle[radius=2pt];
\draw (6.56,1) circle[radius=2pt];
\draw (6.78,0.5) circle[radius=2pt];
\draw (7,0) circle[radius=2pt];  

 \foreach \Point/\PointLabel in {(5,4.5)/P}
        \draw[fill=black] \Point circle (0.05) node[above] {$P$};
 
\foreach \Point/\PointLabel in {(8,4.5)/P}
        \draw[fill=black] \Point node[above] {$\ell_P$};

\foreach \Point/\PointLabel in {(7,0)/P}
        \draw[fill=black] \Point node[below right] {$\ell_3$};

\foreach \Point/\PointLabel in {(5,0)/P}
        \draw[fill=black] \Point node[below right] {$\ell_2$};

\foreach \Point/\PointLabel in {(3,0)/P}
        \draw[fill=black] \Point node[below right] {$\ell_1$};

\end{tikzpicture}

Consider an elementary abelian $2$-subgroup $E$ of order $4$ in $H$, which fixes pointwise a self-polar triangle $T=\{P_1,P_2,P_3\}$ with $P_i \in \ell_1 \cup \ell_2 \cup \ell_3$. Thus, one the following cases occurs, up to relabeling:
\begin{enumerate}
\item $P_1,P_2\in \ell_3$ and $P_3 \in \ell_2$, or
\item $P_i \in \ell_i$ for $i=1,2,3$. 
\end{enumerate}

\begin{tikzpicture}
 \draw (3,4.5) -- (7,4.5);
\draw (9,4.5) -- (13,4.5);
 \draw (5,4.5) -- (5,0);
 \draw (5,4.5) -- (3,0);
 \draw (5,4.5) -- (7,0);
\draw (5,4.5) circle[radius=2pt];
\draw (5,3) circle[radius=2pt];
\draw (5,2.5) circle[radius=2pt];
\draw (5,2) circle[radius=2pt];
\draw (5,1.5) circle[radius=2pt];
\draw (5,1) circle[radius=2pt];
\draw (5,0.5) circle[radius=2pt];
\draw (5,0) circle[radius=2pt];  

\draw (4.32,3) circle[radius=2pt];
\draw (4.1,2.5) circle[radius=2pt];
\draw (3.88,2) circle[radius=2pt];
\draw (3.66,1.5) circle[radius=2pt];
\draw (3.44,1) circle[radius=2pt];
\draw (3.22,0.5) circle[radius=2pt];
\draw (3,0) circle[radius=2pt];  

\draw (5.68,3) circle[radius=2pt];
\draw (5.9,2.5) circle[radius=2pt];
\draw (6.12,2) circle[radius=2pt];
\draw (6.34,1.5) circle[radius=2pt];
\draw (6.56,1) circle[radius=2pt];
\draw (6.78,0.5) circle[radius=2pt];
\draw (7,0) circle[radius=2pt];  

 \foreach \Point/\PointLabel in {(5,4.5)/P}
        \draw[fill=black] \Point circle (0.05) node[above] {$P$};
 
\foreach \Point/\PointLabel in {(7,4.5)/P}
        \draw[fill=black] \Point node[above] {$\ell_P$};

\foreach \Point/\PointLabel in {(7,0)/P}
        \draw[fill=black] \Point node[below right] {$\ell_3$};

\foreach \Point/\PointLabel in {(5,0)/P}
        \draw[fill=black] \Point node[below right] {$\ell_2$};

\foreach \Point/\PointLabel in {(3,0)/P}
        \draw[fill=black] \Point node[below right] {$\ell_1$};

 \draw (11,4.5) -- (11,0);
 \draw (11,4.5) -- (9,0);
 \draw (11,4.5) -- (13,0);
\draw (11,4.5) circle[radius=2pt];
\draw (11,3) circle[radius=2pt];
\draw (11,2.5) circle[radius=2pt];
\draw (11,2) circle[radius=2pt];
\draw (11,1.5) circle[radius=2pt];
\draw (11,1) circle[radius=2pt];
\draw (11,0.5) circle[radius=2pt];
\draw (11,0) circle[radius=2pt];  

\draw (10.32,3) circle[radius=2pt];
\draw (10.1,2.5) circle[radius=2pt];
\draw (9.88,2) circle[radius=2pt];
\draw (9.66,1.5) circle[radius=2pt];
\draw (9.44,1) circle[radius=2pt];
\draw (9.22,0.5) circle[radius=2pt];
\draw (9,0) circle[radius=2pt];  

\draw (11.68,3) circle[radius=2pt];
\draw (11.9,2.5) circle[radius=2pt];
\draw (12.12,2) circle[radius=2pt];
\draw (12.34,1.5) circle[radius=2pt];
\draw (12.56,1) circle[radius=2pt];
\draw (12.78,0.5) circle[radius=2pt];
\draw (13,0) circle[radius=2pt];  

 \foreach \Point/\PointLabel in {(11,4.5)/P}
        \draw[fill=black] \Point circle (0.05) node[above] {$P$};
 
\foreach \Point/\PointLabel in {(13,4.5)/P}
        \draw[fill=black] \Point node[above] {$\ell_P$};

\foreach \Point/\PointLabel in {(13,0)/P}
        \draw[fill=black] \Point node[below right] {$\ell_3$};

\foreach \Point/\PointLabel in {(11,0)/P}
        \draw[fill=black] \Point node[below right] {$\ell_2$};

\foreach \Point/\PointLabel in {(9,0)/P}
        \draw[fill=black] \Point node[below right] {$\ell_1$};

\foreach \Point/\PointLabel in {(9,0)/P}
        \draw[fill=black] \Point node[below right] {$\ell_1$};

\foreach \Point/\PointLabel in {(4.2,-1)/P}
        \draw[fill=black] \Point node[below right] {Case (1)};

\foreach \Point/\PointLabel in {(10.3,-1)/P}
        \draw[fill=black] \Point node[below right] {Case (2)};

\draw (11.68,3) node[anchor=north]{}
  -- (11.9,2.5) node[anchor=north]{}
  -- (11,3) node[anchor=south]{}
  -- cycle;
\foreach \Point/\PointLabel in {(11.68,3)/P}
        \draw[fill=black] \Point node[above right] {$P_1$};
\foreach \Point/\PointLabel in {(11.9,2.5)/P}
        \draw[fill=black] \Point node[above right] {$P_2$};
\foreach \Point/\PointLabel in {(11,3)/P}
        \draw[fill=black] \Point node[above left] {$P_3$};

\draw (5.68,3) node[anchor=north]{}
  -- (4.1,2.5) node[anchor=north]{}
  -- (5,3) node[anchor=south]{}
  -- cycle;

\foreach \Point/\PointLabel in {(4.1,2.5)/P}
        \draw[fill=black] \Point node[above left] {$P_1$};
\foreach \Point/\PointLabel in {(5,3)/P}
        \draw[fill=black] \Point node[above right] {$P_2$};
\foreach \Point/\PointLabel in {(5.68,3)/P}
        \draw[fill=black] \Point node[above right] {$P_3$};
\end{tikzpicture}

Suppose that Case (2) occurs. Then there exists $k$ such that $\sigma^k(P_1)=P_2$. This implies that $\sigma^k(T)=T$ and hence $\sigma^k$ normalizes $E$, a contradiction to $7\nmid|N_{\PGU(3,q)}(E)|$.
Hence Case (1) occurs. Since $H$ contains $14$ elementary abelian $2$-subgroups of order $4$ and just $21$ involutions, there exists a self-polar triangle $T^\prime=\{P^\prime_1,P^\prime_2,P^\prime_3\}\ne T$ fixed by an elementary abelian $2$-subgroup $E^{\prime}\ne E$ such that $|E\cap E^{\prime}|=2$. Hence two vertices of $T$ and $T^\prime$ coincide, say $P_1=P_1^\prime$ and $P_2=P_2^\prime$, while $P_3 \ne P_3^\prime$ as $E\ne E^{\prime}$. This is a contradiction, because $P_3$ is uniquely determined by its polar line $P_2 P_3$.
 This shows that elements of order $7$ in $H$ are not elations, and hence are of type (D).
The claim follows by the Riemann-Hurwitz formula and Theorem \ref{caratteri}.

\textbf{Case $p \ne 7$ and $\sqrt{-7}\notin\mathbb{F}_q$.} The condition $\sqrt{-7}\notin\mathbb{F}_{q}$ implies that $p \equiv 2,5,6 \pmod 7$ and $n$ is odd.
Note that $H$ contains dihedral groups of order $6$ and $8$, proving that if $3 \mid (q+1)$ (resp. $4 \mid (q+1)$) then an element of order $3$ (resp. of order $4$) is of type (B1). 
The claim follows by the Riemann-Hurwitz formula and Theorem \ref{caratteri}.

If $K$ is a proper subgroup of $H$, then either $K$ contains a cyclic normal subgroup of order $3$ or $7$, and hence $K$ fixes a point or a triangle by Remark \ref{rimarco}; or $K$ contains a elementary abelian normal $2$-subgroup of order $4$, and hence $K$ fixes a self-polar triangle.
\end{proof}

\begin{remark}
All the congruences of $p$ modulo $4$ and $3$ listed in Proposition {\rm \ref{psl27}} can occur. Table {\rm \ref{table1}} provides a list of examples.
\end{remark}
\begin{center}
\begin{table}[H]
\begin{small}
\caption{Possible values for $q$ in Proposition \ref{psl27}} \label{table1}
\begin{tabular}{|c|c|c|c|}
\hline $q \equiv_4$ & $q \equiv_3$ & $q \equiv_7$ & Example for $q$ \\
\hline 3 & 1 & 6 & 139\\
\hline 3 & 1 & 3 & 31\\
\hline 3 & 1 & 5 & 19\\
\hline 3 & 2 & 6 & 84\\
\hline 3 & 2 & 3 & 59\\
\hline 3 & 2 & 5 & 47\\
\hline 1 & 1 & 6 & 13\\
\hline 1 & 1 & 3 & 73\\
\hline 1 & 1 & 5 & 61\\
\hline 1 & 2 & 6 & 41\\
\hline 1 & 2 & 3 & 17\\
\hline 1 & 2 & 5 & 5\\
\hline
\end{tabular}
\end{small}
\end{table}
\end{center}

\subsection{$G$ is a subgroup of $SmallGroup(720,765)$, when $q$ is an odd power of $5$}

\begin{proposition} \label{gen720}
Let $q=5^n$ be an odd power of $p=5$ and $H$ be a subgroup of $\PGU(3,q)$ $H \cong {\rm SmallGroup(}$720$,$765${\rm )}$.
If $G$ is a subgroup of $H$ such that $G$ has no fixed points or triangles in $\PG(2,q^6)$, then the genus of $\cH_q/G$ is one of the following:
\begin{equation}\label{720}
\frac{q^2-10q+25}{72},\quad \frac{q^2-16q+55}{120}, \quad \frac{q^2-10q+25}{144},\quad \frac{q^2-46q+205}{720}, \quad \frac{q^2-46q+205}{1440}.
\end{equation}
where $G\cong SmallGroup(36,9)$, $G\cong{\rm A}_5$, $G\cong\PSU(3,2)$, $G\cong{\rm A}_6$, $G=H$, respectively.
Conversely, if $\bar g$ is one of the integers in Equation \eqref{720}, then there exists a subgroup $G$ of $H$ such that $G$ has no fixed points or triangles in $\PG(2,q^6)$ and $g(\cH_q/G)=\bar g$.
\end{proposition}

\begin{proof}
By direct checking, either $G\cong SmallGroup(36,9)$, $G\cong{\rm A}_5$, $G\cong\PSU(3,2)$, $G\cong{\rm A}_6$, $G=H$; or $G$ contains a normal subgroup which is cyclic or elementary abelian of order $4$, and hence $G$ fixes a point or a triangle.
If $G\cong SmallGroup(36,9)$ or $G\cong\PSU(3,2)$, or $G\cong{\rm A}_5$, then $G$ has no fixed points or triangles and the genus of $\cH_q/G$ is computed in Propositions \ref{propohessian} or \ref{propopgl}, respectively.

Elements $\sigma$ of order $3$ are of type (B1); in fact, $H$ contains dihedral subgroups of order $6$ containing $\sigma$, implying that $\sigma$ cannot be neither of type (B3) nor of type (A).
Elements of order $5$ are of type (D), because they are contained in dihedral groups of order $10$.
The genus of $\cH_q/G$ for $G\cong{\rm A}_6$ and $G=H$ can be computed by the Riemann-Hurwitz formula and Theorem \ref{caratteri}.

Finally we note that, if $G=H$ or $G\cong{\rm A}_6$, then $G$ does not fix any point or triangle, because $G$ contains a subgroup isomorphic to ${\rm A}_5$ which fixes no points or triangles.
\end{proof}

\subsection{$G$ is a subgroup of ${\rm A}_6$, when $q$ is an even power of $p=3$, or $5$ is a square in $\mathbb{F}_q$ but $\mathbb{F}_q$ contains no primitive cube roots of unity}

\begin{proposition} \label{a6}
Let $q=p^n$ be a power of an odd prime $p$, where either $p=3$ and $n$ is even, or $\sqrt{5}\in\mathbb{F}_q$ and $\mathbb{F}_q$ contains no primitive cube roots of unity. Let $H$ be a maximal subgroup of $\PSU(3,q)$ isomorphic to the alternating group ${\rm A}_6$.
If $G$ is a subgroup of $H$ such that $G$ has no fixed points or triangles in $\PG(2,q^6)$, then the genus of $\cH_q/G$ is one of the following:
\begin{equation}\label{a6tutto}
\frac{q^2-46q+673-80\alpha-90\beta-144\gamma}{720},\quad \frac{q^2-16q+103-20\alpha-24\gamma}{120},
\end{equation}
\begin{equation}\label{bastardo}
\frac{q^2-10q+61-18\beta}{72},
\end{equation}
where
$$\alpha=\begin{cases} 2, \ if \ p=3, \\ 0, \ otherwise; \end{cases} \ \beta= \begin{cases} 2, \ if \ p=3 \ or \ q \equiv 1 \pmod 4, \\ 0, \ otherwise; \end{cases} \  \gamma=\begin{cases} 0, \ if \ 5 \mid (q+1), \\ 2, \ otherwise,\end{cases}$$
and $G\cong{\rm A}_6$, $G\cong{\rm A}_5$, $G\cong SmallGroup(36,9)$, respectively.
Conversely, if $\bar g$ is one of the integers in Equation \eqref{a6tutto}, or $p\ne3$ and $\bar g$ is the integer in Equation \eqref{bastardo}, then there exists a subgroup $G$ of $H$ such that $G$ has no fixed points or triangles and $g(\cH_q/G)=\bar g$.
\end{proposition}

\begin{proof}
If $G$ has a normal subgroup which is either cyclic or of order $4$, then $G$ fixes a point or a triangle. Hence, we can assume that $G\cong{\rm A}_6$, or $G\cong{\rm A}_5$, or $G\cong SmallGroup(36,9)$.
Also, if $p=3$ and $G\cong SmallGroup(36,9)$, then $G$ fixes a point $P\in\cH_q(\mathbb{F}_{q^2})$ because $G$ has a normal Sylow $3$-subgroup.

Elements of order $3$ are either of type (D) or of type (B1), according to $p=3$ or $p\ne3$, because they are contained in dihedral subgroups of order $6$.
If $p=5$ or $5\mid(q+1)$, then elements of order $5$ are of type (D) or (B1), respectively, because they are contained in dihedral subgroups of order $10$.
If $4\mid(q+1)$, then elements of of order $4$ are of type (B1), because they are contained in dihedral subgroups of order $8$.
Then Equation \eqref{a6tutto} follows from the Riemann-Hurwitz formula together with Theorem \ref{caratteri}.

If $p\ne3$ and $G\cong SmallGroup(36,9)$, then by Proposition \ref{propohessian} $G$ fixes no points nor triangles; if $p\ne3$ and $G\cong{\rm A}_5$ or $G\cong{\rm A}_6$, then $G$ fixes no points or triangles because $G$ contains $SmallGroup(36,9)$.
If $p=3$ and $G\cong{\rm A}_5$, then $G$ cannot fix any point or triangle, since $G$ contains both elements of order $3$ and elementary abelian $2$-subgroups of order $4$. If $p=3$ and $G\cong{\rm A}_6$, then $G$ fixes no points nor triangles as $G$ contains ${\rm A}_5$.
\end{proof}

\subsection{$G$ is a subgroup of ${\rm A}_7$, when $q$ is an odd power of $p=5$}

\begin{proposition} \label{a7gen}
Let $q=5^n$ be an odd power of $p=5$ and let $H$ be a subgroup of $\PGU(3,q)$ isomorphic to the alternating group ${\rm A}_7$. 
If $G$ is a subgroup of $H$ such that $G$ has no fixed points or triangles in $\PG(2,q^6)$, then the genus of $\cH_q/G$ is one of the following:
\begin{equation}\label{a7grossi}
\frac{q^2-106q+2665-720\beta}{5040},\quad \frac{q^2-46q+205}{720},\quad \frac{q^2-22q+229-48\beta}{336},
\end{equation}
\begin{equation}\label{a7piccoli}
 \frac{q^2-26q+105}{240},\quad \frac{q^2-16q+55}{120},\quad \frac{q^2-10q+25}{72},
\end{equation}
where
$$ \beta=\begin{cases} 0, & \textrm{if}\quad 7\mid(q+1), \\ 3, & \textrm{otherwise}, \end{cases} $$
and $G\cong{\rm A}_7$, $G\cong{\rm A}_6$, $G\cong\PSL(2,7)$, $G\cong{\rm A}_5\rtimes C_2$, $G\cong{\rm A}_5$, $G\cong SmallGroup(36,9)$, respectively.
Conversely, if $\bar g$ is one of the integers in Equations \eqref{a7grossi} and \eqref{a7piccoli}, then there exists a subgroup $G$ of $H$ such that $G$ has no fixed points or triangles and $g(\cH_q/G)=\bar g$.
\end{proposition}

\begin{proof}
By direct checking, either $G\cong{\rm A}_7$, $G\cong{\rm A}_6$, $G\cong\PSL(2,7)$, $G\cong{\rm A}_5\rtimes C_2$, $G\cong{\rm A}_5$, $G\cong SmallGroup(36,9)$, or $G$ a normal subgroup which is cyclic or of order $4$ and hence $G$ fixes a point or a triangle.
If $G\cong{\rm A}_6$ or $G\cong\PSL(2,7)$ or $G\cong{\rm A}_5$ or $G\cong SmallGroup(36,9)$, then $G$ has no fixed points or triangles and the genus of $\cH_q/G$ is computed, by Propositions \ref{psl27}, \ref{a6}, and \ref{propohessian}.

Elements of order $3$ (resp. $5$) are of type (B1) because they are contained in dihedral subgroups of order $6$ (resp. $10$).
Elements of order $6$ are of type (B1) because their squares are of type (B1).
If $7\mid(q+1)$, then elements of order $7$ are of type (B1), since they are contained in semidirect products of order $21$ but not in cyclic subgroups of order $21$.
Now Equations \eqref{a7grossi} and \eqref{a7piccoli} follow from the Riemann-Hurwitz formula and Theorem \ref{caratteri}.

If $G\cong{\rm A}_5 \rtimes C_2$ or $G=H$, then $G$ has no fixed points or triangles because $G$ contains ${\rm A}_5$.
\end{proof}

\subsection{$G$ is a subgroup of $\PGU(3,p^m)$ with $m\mid n$ and $n/m$ odd}\label{secPSU}

Let $G$ be a subgroup of $\PGU(3,p^m)$ with $m\mid n$ and $n/m$ odd such that $G$ does not fix any point or triangle in $\PG(2,q^6)$.
By Theorem \ref{MaximalSubgroupsPGU} and Proposition \ref{36ndispari}, either $G\cong SmallGroup(36,9)$ when $q$ is an odd power of $2$; or $G\cong\PSU(3,p^k)$ or $G\cong\PGU(3,p^k)$ where $k\mid m$ and $m/k$ is odd.

\begin{proposition}
Let $q=p^n$ be a power of a prime $p$, and $G$ be a subgroup of $\PGU(3,q)$ such that either
\begin{itemize}
\item[(i)] $q$ is an odd power of $2$ and $G=G_1\cong SmallGroup(36,9)$; or
\item[(ii)] $G=G_2\cong\PGU(3,p^k)$ with $k\mid n$ and $n/k$ odd; or
\item[(iii)] $G=G_3\cong\PSU(3,p^k)$ with $3\mid(q+1)$, $k\mid n$ and $n/k$ odd.
\end{itemize}
Then the genus of $\cH_q/G$ is the following:
\begin{equation}\label{369genere} g(\cH_q/G_1)=\frac{q^2-10q+16}{72};
\end{equation}
\begin{equation}\label{sottoPGUgenere}
g(\cH_q/G_2)= 1+\frac{q^2-q-2-\Delta}{2\bq^3(\bq^3+1)(\bq^2-1)},
\end{equation}
where
$$ \Delta= (\bar q -1)(\bar q^3+1)\cdot(q+2) + (\bar q^3-\bar q)(\bar q^3+1)\cdot2 + \bar q(\bq^4-\bq^3+\bq^2)\cdot(q+1)$$
$$ + (\bq^2-\bq-2)\frac{(\bq^3+1)\bq^3}{2}\cdot2 + (\bq-1)\bq(\bq^3+1)\bq^2\cdot1 + (\bq^2-\bq)\frac{\bq^6+\bq^5-\bq^4-\bq^3}{3}\cdot\gamma, $$
with
$$ \gamma= \begin{cases} 3, & \textrm{if} \quad(\bq^2-\bq+1)\mid(q^2-q+1)\;\textrm{ and }\;\bq\ne2, \\
0, & \textrm{if} \quad(\bq^2-\bq+1)\mid(q+1)\;\textrm{ and }\;\bq\ne2,\\
3, & \textrm{if} \quad \bq=2\;\textrm{and}\;3\mid(q+1)\;\textrm{and}\;3\nmid n,\\
0, & \textrm{if} \quad \bq=2\;\textrm{and}\;3\mid(q+1)\;\textrm{and}\;3\mid n;\\
\end{cases} $$
\begin{equation}\label{sottoPSUgenere}
g(\cH_q/G_3)=\frac{3(q^2-q-2-\Delta)}{2{\bar q}^{3}({\bar q}^{2}-1)({\bar q}^{3}+1)}+1,
\end{equation}
where 
$$ \Delta= (\bar q -1)(\bar q^3+1)\cdot(q+2) + (\bar q^3-\bar q)(\bar q^3+1)\cdot2 + ((\bar q+1)/3-1)(\bq^4-\bq^3+\bq^2)\cdot(q+1)$$
$$ + ((\bar q^2-1)/3-(\bar q+1)/3)\frac{(\bar q^3+1)\bar q^3}{2}
\cdot2 + (\bq-1)((\bq+1)/3-1)(\bq^3+1)\bq^2 \cdot 1 +((\bar q^2-\bar q+1)/3-1)\frac{\bar q^6+\bar q^5-\bar q^4-\bar q^3}{3}\cdot\delta, $$
with 
$$\delta=\begin{cases}3, \ if \ (\bq^2-\bq+1)/3 \mid(q^2-q+1),\\ 0, \ if \ (\bq^2-\bq+1)/3 \mid(q+1). \end{cases}$$
\end{proposition}

\begin{proof}
Suppose that $G=G_1$. Elements of order $3$ are of type (B1), since they are contained in dihedral subgroups of order $6$.
Then Equation \eqref{369genere} follows from the Riemann-Hurwitz formula and Theorem \ref{caratteri}.

Suppose that $G=G_2$. Then the genus of $\cH_q/G_2$ has been computed in \cite[Proposition 5.1]{MZ2}, and Equation \eqref{sottoPGUgenere} follows.

Suppose that $G=G_3$. As pointed out in Lemma \ref{lemma1sottoPSU}, we can assume up to conjugation in $\PGU(3,q)$ that the $G$ is the subgroup of elements $\sigma \in \PSU(3,q)$ such that $\sigma$ is defined over $\mathbb{F}_{\bar q^2}$.
First, we classify the elements of $\PSU(3,\bar q)$ seen as the automorphism group of a Hermitian curve $\cH_{\bar q}$, using the order statistics of $\PSU(3,\bar q)$.
\begin{enumerate}

\item There are exactly $(\bar q -1)({\bar q}^3+1)$ elements of type (C).
In fact, for each $P\in\cH_{\bar q}(\mathbb F_{{\bar q}^2})$, there exist exactly $\bar q-1$ elations in $\PSU(3,\bar q)$ with center $P$.

\item There are exactly $(\bar q^3-\bar q)(\bar q^3+1)$ elements of type (D). In fact, they are the $p$-elements of $\PSU(3,\bar q)$ which are not of type (C).

\item There are exactly $((\bar q+1)/3-1)(\bar q^4-\bar q^3+\bar q^2)$ elements of type (A). In fact, for each $P\in \PG(2,\bar q^2)\setminus\cH_{\bar q}$ there exist exactly $(\bar q+1)/3-1$ homologies in $\PSU(3,\bar q)$ with center $P$.

\item There are exactly $((\bar q^2-1)/3-(\bar q+1)/3)\frac{(\bar q^3+1)\bar q^3}{2}$ elements of type (B2). In fact, for each pair $\{P,Q\}\subset\cH_{\bar q}(\mathbb F_{\bar q^2})$ there exist exactly $(\bar q^2-1)/3-1$ nontrivial elements of $\PSU(3,\bar q)$ fixing $P$ and $Q$, and $(\bar q+1)/3-1$ of them are homologies with axis $PQ$.

\item There are exactly $((\bar q^2-\bar q+1)/3-1)\frac{\bar q^6+\bar q^5-\bar q^4-\bar q^3}{3}$ elements of type (B3). In fact, any point $P\in\cH_{\bar q^2}(\mathbb{F}_{\bar q^6})\setminus\cH_{\bar q}(\mathbb{F}_{\bar q^2})$ determines a unique triangle $\{P,\Phi_{\bar q^2}(P),\Phi_{\bar q^2}^2(P)\}\subset\cH_{\bar q}(\mathbb{F}_{\bar q^6})$ fixed by a Singer subgroup of order $(\bar q^2-\bar q+1)/3$; see \cite{MZ}.

\item There are exactly $(\bq-1)((\bq+1)/3-1)(\bq^3+1)\bq^2$ elements of type (E). 
In fact, consider a pair $\{P,Q\}$ with $P\in\cH_{\bar q}(\mathbb F_{\bar q^2})$ and $Q\in \PG(2,\bar q^2)\cap \ell_P$, where $\ell_P$ is the tangent line to $\cH_{\bar q}$ at $P$. Any element of type (E) fixing $P$ and $Q$ is uniquely obtained as the product of an elation of center $P$ and a homology of center $Q$; thus there exist exaclty $(\bq-1)((\bq+1)/3-1)$ such elements.

\item The remaining $(q^8-3q^7+8q^6-11q^5+9q^4-4q^3)/18=q^3(q-1)(q^2-q+1)(q^2-q+4)/18$ nontrivial elements are of type (B1).
\end{enumerate}
Now we describe the elements in each class (1) -- (7) according to their geometry with respect to $\cH_q$.
\begin{itemize}
\item[(i)] The elements in class (1) are of type (C).
In fact, let $\bar S$ be one of the $\bq^3+1$ Sylow $p$-subgroups of $\PSU(3,\bq)$ and $S$ be the Sylow $p$-subgroup of $\PSU(3,q)$ containing $\bar S$.
Note that $S$ is a trivial intersection set, since $\cH_q$ has zero $p$-rank; see \cite[Theorem 11.133]{HKT}.
Consider the explicit representation of $S$ given in \cite[Section 3]{GSX}, where $\cH_q$ has norm-trace equation and $S$ fixes the point at infinity.
By direct computation, there are $\bq$ elements of $\bar S$ in the center of $S$. Thus, $\bar S$ has exactly $\bq-1$ elements of type (C); see \cite[Equation (2.12)]{GSX}.
\item[(ii)] The elements in class (2) are of type (D).
In fact, with $\bar S$ as in Case (i), the claim follows from Case (i) counting the remaining nontrivial elements of $\bar S$.

\item[(iii)] Let $\sigma\in \PGU(3,\bq)$ be in class (3). 
Then $\sigma$ is contained in the pointwise stabilizer $D$ of a self-polar triangle with respect to $\cH_\bq$ and $D$ is an abelian group of order $(\bar q+1)^2/3$; see \cite{M}.
Let $C_{q+1}\times C_{q+1}\leq \PGU(3,q)$ be the poitwise stabilizer of a self-polar triangle $\mathcal{T}$ with respect to $\cH_q$, such that $D\leq (C_{q+1}\times C_{q+1})\cap \PSU(3,q)$ which is an abelian group of order $(q+1)^2/3$.

Up to conjugation, $\cH_q$ has Fermat equation $X^{q+1}+Y^{q+1}+T^{q+1}=0$ and $\mathcal{T}$ is the fundamental triangle, so that
$$ D=\{(X,Y,T)\mapsto(\lambda X,\mu Y,T)\mid \lambda^{\bq+1},\mu^{\bq+1}=1\} \cap \PSU(3,q). $$

By direct computation, $D$ contains $3((\bq+1)/3-1)$ elements of type (A) and $(\bq^2-\bq+4)/3$ elements of type (B1).

Since $\PSU(3,\bq)$ is transitive on $\PG(2,\bq^2)\setminus\cH_{q}$, there are exactly $\frac{\left|\PSU(3,\bq)\right|}{6|D|}$ self-polar triangles $\mathcal{T}^\prime$ with respect to $\cH_q$, whose pointwise stabilizer $D^\prime$ is conjugated to $D$ under $\PSU(3,\bq)$.

Note that $D$ and $D^\prime$ intersect non-trivially if and only if $\mathcal{T}$ and $\mathcal{T}^\prime$ have a vertex $P$ in common; in this case, $(D\cap D^\prime)\setminus\{id\}$ is made by $(\bq+1)/3-1$ homologies with center $P$.
The number of points in $\PG(2,\bq^2)\setminus\cH_{q}$ lying on the polar of $P$ is $\bq^2-\bq$; hence, the number of $D^\prime$ which intersect $D$ non-trivially is $((\bq+1)/3-1)(\bq-1)/2$.

Therefore, by direct computation, the number of elements of type (A) in $\PSU(3,\bq)$ is exactly the number $((\bq+1)/3-1)(\bq^4-\bq^3+\bq^2)$ of elements in class (3), and the remaining $(q^8-3q^7+8q^6-11q^5+9q^4-4q^3)/18$ elements of the subgroups of $\PSU(3,\bq)$ conjugated to $D$ are of type (B1) and in class (7).

\item[(iv)] Let $\sigma\in \PSU(3,\bq)$ be in class (4).
Since the order $o(\sigma)$ of $\sigma$ divides $\bq^2-1$ but not $\bq+1$, we have that $o(\sigma)$ divides $q^2-1$ but not $q+1$, as $q$ is an odd power of $\bq$.
Therefore $\sigma$ is of type (B2).

\item[(v)] Let $\sigma\in \PSU(3,\bq)$ be in class (6).
Since the order of $\sigma$ is $p\cdot d$ where $d>1$ and $p\nmid d$, $\sigma$ is of type (E).

\item[(vi)] Let $\sigma\in \PSU(3,\bq)$ be in class (5).
By direct checking, the order $o(\sigma)$ of $\sigma$ divides either $q^2-q+1$ or $q+1$.
Note that $(\bq^2-\bq+1)/3$ is not divisible by $3$. If $(\bq^2-\bq+1) \mid (q^2-q+1)$ then every $\sigma$ in class (5) is of type (B3), by \cite{MZ}. 

If $(\bq^2-\bq+1) \mid (q+1)$ then $\sigma$ is either of type (A) or (B1). We note that $\sigma$ is contained in the maximal subgroup (iv) of $\PSU(3,\bq)$ in \cite{M}, which is a semidirect product $C_{(\bq^2-\bq+1)/3} \rtimes C_3$, where $\sigma \in C_{(\bq^2-\bq+1)/3}$ and $C_3$ does not commute with any subgroup of $C_{(\bq^2-\bq+1)/3}$. Therefore $\sigma$ cannot be of type (A), because otherwise every element which normalizes $\sigma$ should commute with $\sigma$.
Then $\sigma$ of type (B1).
\end{itemize}
Now Equation \eqref{sottoPSUgenere} follows from the Riemann-Hurwitz formula and Theorem \ref{caratteri}.
\end{proof}

\section{The geometry of subgroups of $\PGU(3,q)$}\label{sec:geometria}

The following result lists the subgroups of $\PGU(3,q)$ which do not fix a point or a triangle in $\PG(2,q^6)$.
It follows as a corollary of Theorem \ref{MaximalSubgroupsNotFixingPSU} and the analysis ruled out in Section \ref{sec:sottogruppimassimaliPGU}.

\begin{theorem}\label{SubgroupsNotFixingPGU}
Let $q=p^n$ be a prime power. The subgroups of $\PGU(3,q)$ which do not fix a point or a triangle are the following.
\begin{enumerate}
\item $\PSU(3,p^m)$, when $m\mid n$ and $n/m$ is odd.
\item $\PGU(3,p^m)$, when $m\mid n$ and $n/m$ is odd.

{\rm For $p>2$, also the following subgroups.}
\item The Hessian groups of order $216$, $72$, and $36$; they are isomorphic to $\PGU(3,2)$, $\PSU(3,2)$, and $SmallGroup(36,9)$, respectively.
\item $\PGL(2,q)$ fixing a conic.
\item $\PSL(2,7)$, when $p=7$ or $-7$ is not a square in $\mathbb{F}_{q}$.
\item The alternating group $A_6$, when $p=3$ and $n$ is even, or $5$ is a square in $\mathbb{F}_q$ and $\mathbb{F}_q$ contains no primitive cube roots of unity.
\item A group of order $720$ containing the alternating group $A_6$ when $p=5$ and $n$ is odd, isomorphic to $SmallGroup(720,765)$ in the GAP notation.
\item The alternating group $A_7$, when $p=5$ and $n$ is odd.

{\rm When $q$ is an odd power of $2$, also a group of order $36$ isomorphic to $SmallGroup(36,9)$ in the GAP notation.}
\end{enumerate}
\end{theorem}

\begin{proof}
Let $H$ be a proper subgroup of $\PGU(3,q)$ which does not leave invariant any point or triangle. Then from Theorem \ref{MaximalSubgroupsNotFixingPSU} and Remarks \ref{rimarco2} and \ref{rimarco3}, $H \cap \PSU(3,q)$ is isomorphic to one of the following groups.
\begin{itemize}
\item[-] For $p>2$:
\item the Hessian groups of order $216$ when $9\mid(q+1)$, and of order $72$ or $36$ when $3\mid(q+1)$; or
\item[(i)] ${\rm PGL}(2,q)$ preserving a conic; or
\item[(ii)] ${\rm PSL(2,7)}$ when $p=7$ or $-7$ is not a square in $\mathbb{F}_q$; or
\item[(iii)] $SmallGroup(720,765)$ when $p=5$ and $n$ is odd; or
\item[(iv)] ${\rm A}_6$ when either $p=3$ and $n$ is even, or $5$ is a square in $\mathbb{F}_q$ but $\mathbb{F}_q$ contains no primitive cube roots of unity; or
\item[(v)] ${\rm A}_7$ when $p=5$ and $n$ is odd; or
\item[(vi)] $\PSU(3,p^m)$ with $m\mid n$ and $n/m$ odd; or
\item[(vii)] $\PGU(3,p^m)$ with $m\mid n$, $n/m$ is odd, and $3$ divides both $n/m$ and $q+1$.
\item[-] For $p=2$:
\item[(viii)] $\PSU(3,p^m)$ with $m\mid n$ and $n/m$ odd; or
\item[(ix)] $\PGU(3,p^m)$ with $m\mid n$, $n/m$ is odd, and $3$ divides both $n/m$ and $q+1$; or
\item[(x)] $SmallGroup(36,9)$ when $n$ is odd.
\end{itemize}
Also, $H$ is contained in a maximal subgroup $M$ of $\PGU(3,q)$. From Theorem \ref{MaximalSubgroupsPGU}, $M\cong\PSU(3,q)$ unless one of the following cases holds: $H \cap \PSU(3,q)$ is isomorphic to the Hessian groups of order $72$ or $36$, and $M \cong H_{216}$; or $H \cap \PSU(3,q) \cong \PSU(3,p^m)$ where $m \mid n$ and $n \mid m \ne 3$ is odd, and $M \cong \PGU(3,p^m)$.

If $M\cong\PSU(3,q)$, then $H$ is isomorphic to one of the groups {\it (i)--(x)} from Theorem \ref{MaximalSubgroupsNotFixingPSU} and Remark \ref{rimarco3}.

If $H \cap \PSU(3,q) \cong \PSU(3,p^m)$, then either $H=M\cong\PGU(3,p^m)$ or $H \cong \PSU(3,p^m)$, since there are no proper subgroups of $\PGU(3,p^m)$ properly containing $\PSU(3,p^m)$.

If $H\cap\PSU(3,q)\cong H_{72}$, then either $H=M\cong H_{216}$ or $H \cong H_{72}$, since there are no proper subgroups of $H_{216}$ properly containing $H_{72}$.

If $H\cap\PSU(3,q)\cong H_{36}$, then $H=M\cong H_{216}$ or $H \cong H_{36}$, since the unique proper subgroup of $H_{216}$ properly containing $H_{36}$ is $H_{72}$, which is contained in $\PSU(3,q)$.
\end{proof}

\end{document}